\documentclass[letter,11pt, reqno]{amsart} 
\usepackage{amsmath}
\usepackage{amssymb}
\usepackage{amsthm}
\usepackage{mathtools}
\usepackage[margin=1in]{geometry}
\usepackage{hyperref}
\usepackage[style=numeric, firstinits=true, doi=false, isbn=false, url=false]{biblatex}
\renewbibmacro{in:}{}
\DeclareFieldFormat
  [article,inbook,incollection,inproceedings,patent,thesis,unpublished]
  {title}{#1}
\usepackage{setspace}
\usepackage{dsfont} 

\newtheorem{theorem}{Theorem}[section]
\newtheorem{lemma}[theorem]{Lemma}
\newtheorem{corollary}[theorem]{Corollary}
\newtheorem{prop}[theorem]{Proposition} 

\numberwithin{equation}{section} 

\DeclarePairedDelimiter\abs{\lvert}{\rvert}
\DeclarePairedDelimiterX{\inn}[2]{\langle}{\rangle}{#1, #2}
\DeclarePairedDelimiterX{\norm}[1]{\lVert}{\rVert}{#1}
\DeclarePairedDelimiterX{\Bnorm}[1]{\Big\lVert}{\Big\rVert}{#1}

\newcommand*{\one}{\mathds{1}}

\newcommand{\RR}{{\mathbb{R}}}
\newcommand{\CC}{{\mathbb{C}}}
\newcommand{\TT}{{\mathbb{T}}}

\newcommand{\NN}{{\mathbb{N}}} 

\newcommand{\EE}{{\mathbb{E}}}

\newcommand{\nn}{\tilde{n}}

\newcommand{\T}[1]{\text{#1}}

\DeclareMathOperator{\arccosh}{arccosh}

\let\epsilon\varepsilon 
\let\pphi\phi
\let\phi\varphi 
\let\varphi\pphi 

\begin{document}
\title{On the regularity conditions in the CLT for the LUE}
\author{Henry Hu}

\address{University of Toronto, Canada}
\email{henryw.hu@mail.utoronto.ca}
\begin{abstract} 
   We consider the Laguerre Unitary Ensemble (LUE), the set of $n\times n$ sample covariance matrices $M = \frac{1}{n}X^*X$ where the $m\times n$ ($n \le m$) matrix $X$ has i.i.d. standard complex Gaussian entries. In particular we are concerned with the case where $\alpha := m - n$ is fixed, in which case the limiting eigenvalue density has a hard edge at $0$. We study the minimal regularity conditions required for a central limit theorem (CLT) type result to hold for the linear spectral statistics of the LUE. As long as the expression for the limiting variance is finite (and a slightly stonger condition holds near the soft edge) we show that the variance of the linear spectral statistic converges and consequently the CLT holds. Our methods are based on analyzing the explicit kernel for the LUE using asymptotics of the Laguerre polynomials. The CLT follows from approximating the test function of the statistic by Chebyshev polynomials.
\end{abstract}
 
\maketitle

\section{Introduction and main results}
Given a random matrix ensemble $(M_n)_{n\in \NN}$, where $M_n$ are square matrices of size $n$, we can study the asymptotic behavior of its eigenvalues $\{\lambda_k\}_{k\le n}$. In particular we are interested in the fluctuations of the linear spectral statistics (LSS) associated to $f:\CC \to \CC$,
  \begin{equation*}
   \mathcal{N}_n[f] := \sum_{j=1}^n f(\lambda_j).
  \end{equation*} 
When the ensemble consists of matrices with only real eigenvalues, we will restrict ourselves to the case of $f: \RR \to \RR$. The associated centred linear spectral statistics is given by,
  \begin{equation} \label{defs:CenteredStat}
  \mathcal{N}_n^\circ[f] := \mathcal{N}_n[f] - \EE[\mathcal{N}_n[f]].
  \end{equation}
These statistics have been extensively studied and exhibit many amazing mathematical properties. They have also shown up in a variety of other fields, with applications in physics, statistics, big data, wireless communications, etc. (see surveys \cite{Edelman2013, Priya2022} and references therein). 

The first notable instance of LSS in the literature dates back to the work of physicist Eugene Wigner in the 1950s \cite{Wigner1955}. His motivation stemmed from observations made about experimental data in nuclear physics; he postulated that the gaps in the spectrum of large nuclei should correspond to the spacings between the eigenvalues of random matrices. A well-known result of his is a type of law of large numbers for the density of the eigenvalues -- the Wigner semicircle law. Wigner matrices are the $n\times n$ complex Hermitian matrices with i.i.d entries of variance $\EE[\abs{h_{jk}}^2] = 1/n$ above the diagonal and i.i.d entries of variance $\EE[h_{jj}^2] = 1/n$ on the diagonal. In particular, the Gaussian Unitary Ensemble (GUE) is the complex Hermitian Wigner ensemble with entries that are i.i.d centred complex Gaussian above the diagonal and i.i.d centred real Gaussian on the diagonal. Then the Wigner semicircle law states for Wigner matrices and smooth $f: \RR\to \RR$,
  \begin{equation*}
   \lim_{n\to\infty} \frac{1}{n} \mathcal{N}_n[f] = \int_\RR f(x)\rho_{sc}(x) dx
  \end{equation*}
almost surely, where 
  \begin{equation*}
    \rho_{sc}(x) := \frac{1}{2\pi}\sqrt{\max(4-x^2, 0)}.
  \end{equation*}
We note that for the Wigner matrices there are two soft edges, at $x = \pm 2$, in the eigenvalue density, as the density tends to $0$. 

A closely related type of random matrices are the random sample covariance matrices, first studied by Wishart in the 1920s \cite{Wishart1928} as a tool in statistics to estimate large covariance matrices. Random sample covariance matrices are square matrices of the form
  \begin{equation} \label{eq:coVarDefs}
   M_n = \frac{1}{n}X^*X,
  \end{equation}
where $X$ is a random matrix of size $m\times n$, $n \le m$. The entries of $X$ satisfy the variance condition $\EE[\abs{h_{jk}}^2] = 1$. An important example is the Laguerre Unitary Ensemble (LUE), where $X$ has entries which are i.i.d. standard complex Gaussian. Notably, the random sample covariance matrices also obey a law of large numbers -- the famous Marchenko-Pastur law \cite{Marchenko1967}. It states that when the large-$n$ limit of the ratio between $n$ and $m$ exists, i.e.
  \begin{equation*}
    d := \lim_{n\to\infty} \frac{n}{m}
  \end{equation*}
then for smooth $f:\RR\to \RR$: 
  \begin{equation*}
    \lim_{n\to\infty}\frac{1}{n}\mathcal{N}_n[f] = \int_\RR f(x) \rho_{MP}dx
  \end{equation*}
almost surely, where
  \begin{equation*}
    \rho_{MP}(x) := \frac{1}{2\pi d} \frac{\sqrt{\max((1+\sqrt{d} - x)(x-1+\sqrt{d}), 0)}}{x}.
  \end{equation*}
In particular, when $d = 1$ the eigenvalue density has a hard edge at $x = 0$, i.e. goes to $\infty$. We call the support of the limiting eigenvalue density the bulk. In the case of $d =1$ this corresponds to the interval $[0, 4]$. 

Interest developed in the fluctuations of the LSS when Arkharov \cite{Arkharov1971} discovered in the 1970s that the centered LSS (of sample covariance matrices) converges to a Gaussian. The remarkable phenomena is that unlike i.i.d random variables, the LSS does not need a $n^{-\frac{1}{2}}$ scaling for a CLT-type result to hold. This is due to the strongly-correlated nature of the eigenvalues. It is unlikely for two neighboring eigenvalues to be very far or very close to each other, relative to the expected gap. This property is described as the rigidity of the spectrum. Results concerning the rigidity of the spectrum of the Wigner and sample covariance matrices can be found in \cite{Erdos2012, Pillai2014}. Arkharov's discovery was then proven rigorously by Jonsson in the 1980s \cite{Jonsson1982} and further improved by Girko \cite{Girko2012}; Girko was also the first to give explicit formulas for the limiting variance, albeit in more complex form than more current results. Since then there have been a plethora of results proving the CLT for various classes of functions and random matrices (i.e. see \cite{Diaconis2001, Brian2007} and other references below). 

It is of interest to understand the minimal regularity conditions on $f$ for $\mathcal{N}_n^\circ[f]$ to converge to a Gaussian (see e.g. \cite{Johansson1998, Landon2022, Lytova2009, Shcherbina2011, Sosoe2013} and the references therein). Among the earliest of this kind of result is Johansson \cite{Johansson1998} who proved a CLT for unitarily invariant matrices, including the GUE, with $f \in H^{2+\epsilon}$. Johansson has further conjectured that the sole condition for the CLT to hold should be that the limiting variance if finite. While showing this conjecture in full generality for general matrix ensembles is out of reach, there have been results which are close for a specific or class of ensembles. In the work \cite{Diaconis2001}, it was shown that the CLT holds for $L^2(\TT)$ test functions with the Circular Unitary Ensemble (CUE), the set of unitary matrices $(M_n)_{n \in \NN}$ with distribution given by the Haar measure on unitary groups, whenever the limiting variance is finite. 
Most relevent to our work is the preprint of Kopel \cite{Kopel2015} which investigated minimal regularity conditions for the CLT to hold for the GUE. Kopel's paper shows, for test functions supported in the bulk, that the variance of the LSS converges iff the expression for the limiting variance is finite. However, the treatment of the soft edges appears incomplete \footnote{Private communications with Benjamin Landon.}. 

The typical method of showing the CLT involves proving estimates or bounds on the variance of $\mathcal{N}_n^\circ[f]$. For random matrix ensembles with determinantal $N$-point correlation functions, the variance is given by (\cite{Pastur2011}, 1.1.39):
  \begin{equation} \label{eq:varKernalRep}
  \begin{split}
    \T{Var}(\mathcal{N}_n^\circ[f]) & = \frac{1}{2}\iint \abs{f(\lambda_1) - f(\lambda_2)}^2 K_n(\lambda_1, \lambda_2)^2 d\lambda_1 d\lambda_2,
  \end{split}
  \end{equation}
where,
  \begin{equation*}
   K_n(x, y) := \sum_{k=0}^{n-1} \psi_k(nx)\psi_k(ny)
  \end{equation*}
for suitable functions $\psi_k$. In the case of the GUE, the $\psi_{k}$ for $K_n(x, y)$ in the variance formula \eqref{eq:varKernalRep} takes the form of normalized Hermite polynomials. For the LUE, $\psi_k^{(\alpha)}$ are the normalized Laguerre polynomials \eqref{info:weights}, with $\alpha = m-n$ \eqref{eq:coVarDefs}. The limiting variance for the LUE is (\cite{Pastur2011}, Remark 4.1):
  \begin{equation}\label{eq:VLUE}
      V_{\T{LUE}}[f] := \int_{0}^4 \int_{0}^4\Big(\frac{f(x)-f(y)}{x-y}\Big)^2 \frac{4-(x-2)(y-2)}{4\pi^2 \sqrt{4-(x-2)^2}\sqrt{4-(y-2)^2}} dx dy.
  \end{equation}
We note the limiting expression for the variance of the GUE is similar \eqref{eq:GUEVar}. The main takeaway of these expressions is that the limiting variances of the GUE and LUE share a common feature. They both have a singular part in the integral that is carried through from $\T{Var}(\mathcal{N}_n^\circ[f])$: 
  \begin{equation}\label{eq:FDefs}
   F(x, y):= \Big(\frac{f(x) - f(y)}{x-y}\Big)^2.
  \end{equation}
In particular, this is the integrand of the $\dot{H}^{\frac{1}{2}}(\RR)$-seminorm. We recall that the $\dot{H}^s(\RR)$-seminorm is defined:
  \begin{equation} \label{eq:HhalfNorm}
   \norm{f}_{\dot{H}^s} := \int_\RR \abs{\xi}^{2s}\abs{\hat{f}(\xi)}^2 d\xi = c\iint_{\RR^2} \frac{(f(x) - f(y))^2}{(x-y)^{2s + 1}} dx dy,
  \end{equation}
where $\hat{f}$ denotes the Fourier transform and $c > 0$ is a suitable constant. In particular, the space $H^s(\RR)$ consists of $L^2$-functions with finite $\dot{H}^s$-seminorm. From this we see Landon-Sosoe's \cite{Landon2022} result for (smooth) real symmetric Wigner matrices, where the test function $f$ only needs to be $H^{\frac{1}{2} + \epsilon}(\RR)$ and supported in the bulk $(-2 + \delta, 2-\delta)$, is approaching Johansson's conjecture. The term $F(x, y)$ also shows up more generally in the limiting variance of sample covariance ensembles \cite{Pastur2011}.

To see where the $x-y$ is coming from, we note that similar to the GUE and other exactly solvable ensembles, the kernel $K_n(x, y) = K_n^{(\alpha)}(x, y)$ (where $\alpha = m-n$) of the LUE admits a Christoffel--Darboux formula \cite{Forrester1993}:
  \begin{equation} \label{eq:ChrisDarb}
  \begin{split}
    K_n^{(\alpha)}(x, y) = \sum_{k=0}^{n-1} \psi_k^{(\alpha)}(nx)\psi_k^{(\alpha)}(ny) = \sqrt{n(n+\alpha)} \frac{\psi_{n-1}^{(\alpha)}(nx)\psi_{n}^{(\alpha)}(ny) - \psi_{n}^{(\alpha)}(nx)\psi_{n-1}^{(\alpha)}(ny)}{x-y},
  \end{split}
  \end{equation}
where $\psi_k^{(\alpha)}$ are the normalized Laguerre polynomials \eqref{info:weights} (hence the name). 

In this work, we are primarily interested in the fluctuations of the LSS for the LUE when $d = 1$. Our work improves substantially on all existing regularity results for the CLT in random matrix theory. In particular, the LUE with $d = 1$ has both a soft edge and a hard edge, and is therefore more technically challenging than the GUE and CUE studied by others. Secondly, for test functions in the bulk, our work improves on the results of Landon-Sosoe \cite{Landon2022} by removing the $+\epsilon$ in their $H^{\frac{1}{2} + \epsilon}$ condition. The main results of the paper is as follows:
\begin{theorem} \label{thm:VarConvergence}
  Let $\mathcal{N}_n^\circ[f]$ be the centred LSS \eqref{defs:CenteredStat} of the LUE with $m=n+\alpha$ where $\alpha\in \NN$ is fixed. Suppose $f\in L^\infty([0,\infty))$ and there exists $\epsilon > 0$ such that
  \begin{equation} \label{eq:Assumption1}
    V_{\T{LUE}}^\epsilon[f] := \int_0^{4+\epsilon}\int_0^{4+\epsilon} F(x, y)\frac{1}{8\pi^2}\Big(\frac{\sqrt{\abs{(4-x)y}}}{\sqrt{\abs{(4-y)x}}} + \frac{\sqrt{\abs{(4-y)x}}}{\sqrt{\abs{(4-x)y}}}\Big)dx dy < \infty.
  \end{equation}
Then,
  \begin{equation*}
    \lim_{n\to\infty} \T{Var}(\mathcal{N}_n^\circ[f]) = V_{LUE}[f] := \int_0^4 \int_0^4 F(x, y) \Xi(x, y)dx dy.
  \end{equation*}
Where $F(x, y)$ is defined in \eqref{eq:FDefs} and,
    \begin{equation} \label{eq:XiDefs}
     \Xi(x, y) := \frac{4-(x-2)(y-2)}{4\pi^2 \sqrt{4-(x-2)^2}\sqrt{4-(y-2)^2}}.
    \end{equation}
\end{theorem} 

\begin{lemma} \label{lem:VLUEApprox}
  Suppose $f$ satisfies \eqref{eq:Assumption1}. Then there exists a sequence of smooth functions $f^n$ in $L^\infty([0,4])$ which approximates $f$ in the $V_{\T{LUE}}$ seminorm. i.e. given $\epsilon > 0$ there exists $N$ such that $n > N$ implies
    \[V_{\T{LUE}}[f-f^n] < \epsilon.\]
\end{lemma}

\begin{corollary}\label{cor:CLT}
  Let $\mathcal{N}_n^\circ[f]$ be the centred LSS \eqref{defs:CenteredStat} of the LUE with $m=n+\alpha$ where $\alpha\in \NN$ is fixed. Suppose $f\in L^\infty([0,\infty))$ such that there exists $\epsilon > 0$ so that \eqref{eq:Assumption1} is satisfied. Then the centred LSS, $\mathcal{N}_n^\circ[f]$, converges in distribution, as $n \to \infty$, to a Gaussian with mean zero and variance $V_{LUE}[f]$ \eqref{eq:VLUE}. 
\end{corollary}
\begin{proof}
 This follows from the pointwise convergence of the characteristic functions. This is done by using Theorem \ref{thm:VarConvergence}, approximating $L^\infty$ functions with Chebyshev polynomials using Lemma \ref{lem:VLUEApprox}, and applying Theorem 4.2 of \cite{Pastur2011}.  
\end{proof}

The condition \eqref{eq:Assumption1} is very close to only assuming the limiting variance is finite. In fact when $x, y \le 4$, we have:
  \[\frac{1}{8\pi^2}\bigg(\frac{\sqrt{\abs{(4-x)y}}}{\sqrt{\abs{(4-y)x}}} + \frac{\sqrt{\abs{(4-y)x}}}{\sqrt{\abs{(4-x)y}}}\bigg) = \frac{4-(x-2)(y-2)}{4\pi^2 \sqrt{4-(x-2)^2}\sqrt{4-(y-2)^2}} = \Xi(x, y).\]
The conjecture of Johansson would imply that there is no need for any condition on the function outside of the bulk. However, the normalized Laguerre polynomials, while decaying exponentially, are not $0$ outside the support of the spectrum. It is only natural to expect that some condition is necessary outside the spectrum; our condition  \eqref{eq:Assumption1} essentially amounts to extending the variance functional outside the spectrum in a natural way. Due to the exponential decay, it is possible that the condition could be marginally improved but we do not pursue this here.

A consequence of Corollary \ref{cor:CLT} is that for the LUE and any bounded function, the hard edge does not contribute any extra noise. This is unlike the case of the Ginibre ensemble, which observes a $H^{1}$-noise on the unit disk and an independent $H^{1/2}$ noise on the unit circle \cite{Brian2007}. The Ginibre ensemble is the set of $n\times n$ random complex matrices with entries which are i.i.d complex Gaussians with variance $1/n$. Its limiting empirical distribution of eigenvalues is the uniform distribution on the unit disk.

We now give a bit more discussion on the implications and background of our result. The commonality of the limiting variance of many ensembles is the singular term $F(x, y)$ in \eqref{eq:FDefs}. It is also where most of the difficulties arise when trying to improve the regularity conditions. When the test function $f$ is sufficiently regular, this term does not pose any difficulties. Lytova-Pastur's result \cite{Lytova2009}, showed the CLT holds for both the GUE and LUE, among others, when the test function $f$ is bounded with bounded derivative, and obtained the same limiting variance as \eqref{eq:VLUE}. Their regularity condition implies that $F(x, y)$ is uniformly bounded in $x, y \in [0, \infty)$. Shcherbina's result \cite{Shcherbina2011} on the closely related Laguerre Orthogonal Ensemble, where the entries of $X$ \eqref{eq:coVarDefs} are real instead of complex, needed the regularity condition of $H^{\frac{3}{2} + \epsilon}(\RR)$. In fact, most previous results about the LUE and the more general random sample covariance matrices assumed that $f$ was sufficiently regular so that $F(x, y)$ is bounded, or at minimum that $F(x, y)$ is well-defined everywhere (i.e. \cite{Bai2010, Najim2016}). The function $F(x, y)$ also appears in the limiting variance of $\beta$-ensembles \cite{Johansson1998}, among others. Terms of similar structure also show up in the limiting variance of other ensembles, such as the CUE \cite{Diaconis2001}, being a scaling of the $\dot{H}^\frac{1}{2}(\TT)$-seminorm:
  \begin{equation*}
    V_{\T{CUE}}[f] := \frac{1}{16\pi^2} \int_0^{2\pi} \int_0^{2\pi} \Big(\frac{f(\phi) - f(\theta)}{\sin(\frac{\phi-\theta}{2})}\Big)^2 d\theta d\phi.
  \end{equation*}

In our proof for Theorem \ref{thm:VarConvergence}, we are able to relax the regularity conditions on the test function $f$ by exploiting exact formulas and asymptotics for $\psi_k^{(\alpha)}$. We show that the contribution to the limiting variance from $\abs{x - y} < \delta$, for small $\delta > 0$, is insignificant. We expect that with appropriate asymptotics, the results should extend to the case for the LUE where $0 <d < 1$ as well as the GUE and possibly other ensembles with exact $N$-point correlation functions.

\subsection{Methodology} We now outline the approach used in this paper. The details of the proofs for Theorem \ref{thm:VarConvergence} and Lemma \ref{lem:VLUEApprox} are found in Sections \ref{sec:Thm1Proof} and \ref{sec:Lem1Proof} respectively. The sketch of the proofs are as follows. For Theorem \ref{thm:VarConvergence}, we split our domain $[0, \infty)^2$ into several (overlapping) pieces,
  \begin{equation*}
    \Omega_1^{\epsilon_1} = [0, 4+\epsilon_1)^2 ,\ \ \Omega_2^{\epsilon_2} = [4+\epsilon_2, \infty)^2,\ \ \Omega_3 = [0, \infty)^2 \setminus (\Omega_1^{\epsilon_1} \cup \Omega_2^{\epsilon_2}),
  \end{equation*}
 where $\epsilon_i > 0$, $\epsilon_2$ is much less than $\epsilon_1$, and showed convergence on each of them separately. This is done in four steps:
  \begin{enumerate}
   \item On $\Omega_3$ and whenever $\abs{x - y}$ is uniformly bounded away from $0$, $F(x, y)$ is uniformly bounded in $x, y$. In particular, this allows us to directly use bulk asymptotics, i.e. $x \in [\delta, 4-\delta]$, for the Laguerre polynomials to obtain convergence results for $F(x, y) \in L^\infty([\delta, 4-\delta]^2)$. Extending the convergence to $F(x, y)\in L^\infty([0, \infty)^2)$ essentially amounts to bounding the contribution from the region outside $[\delta, 4-\delta]^2$ using the $L^\infty$-norm of $F(x, y)$. 
   \item On $\Omega_1^{\epsilon_1}$, we use universal asymptotics for the soft and hard edge to show that $\abs{\Phi_n^{(\alpha)}(x, y)}$ is bounded by $C\Xi(x, y), C>0$. The majority of the work for showing this bound is to deal with the case when both $x$ and $y$ are near the soft or hard edge of the spectrum. In particular, we bound $(x-y)^2K_n^{(\alpha)}(x, y)^2$ by derivatives of Airy and Bessel functions.
   \item On $\Omega_2^{\epsilon_2}$, we show that $K_n^{(\alpha)}(x, y)^2$ itself was $L^1$, allowing us to bypass the problem of $F(x, y)$ blowing up. Returning to the variance formula \eqref{eq:varKernalRep} finishes the proof for $\Omega_2^{\epsilon_2}$. 
   \item Theorem \ref{thm:VarConvergence} follows from combining the three results above with $\Omega_1^{\epsilon}$ and $\Omega_2^{\epsilon/2}$ for sufficiently small $\epsilon > 0$.
  \end{enumerate}
The approach for the region $\Omega_3$ is inspired by Kopel's preprint \cite{Kopel2015}. However the approach used near the edges is entirely new. Lemma \ref{lem:VLUEApprox} follows from first relating $V_{\T{LUE}}$ \eqref{eq:VLUE} to $V_{\T{GUE}}$ \eqref{eq:GUEVar} and then exploiting the well-known connection \cite{Johansson1998} between $V_{\T{GUE}}$ and Chebyshev polynomials. Essentially we show that $V_{\T{GUE}}$ is norm-equivalent to the $\dot{H}^\frac{1}{2}$-seminorm in the the Chebyshev basis -- see Lemma \ref{lem:seminormsEquiv}.

\subsection{On notation and convention}

  \begin{enumerate}
   \item In general, $N$ can vary from line to line. We will also generally omit `there exists sufficiently large $N$ such that for all $n > N$' and simply write `for $n > N$'. Likewise, constants denoted by $C$ will vary from line to line -- in cases where we need to differentiate them we will index them $C_1$, etc. By `positive' constant we mean $C > 0$.
   \item By `$a$ is much less than $b$' we mean that $0 \le a < \frac{b}{100}$. Also, $\alpha$ always denotes a fixed non-negative integer. We use $\nn$ to denote anything of form $n-k$ where $k$ is some fixed non-negative integer that doesn't depend on $n$.
   \item Links to definitions: $\mathcal{N}_n^\circ[f]$ -- \eqref{defs:CenteredStat}; LUE -- \eqref{eq:coVarDefs}; $K_n^{(\alpha)}$ -- \eqref{eq:ChrisDarb};  $V_{\T{GUE}}$ -- \eqref{eq:GUEVar}; $F(x, y)$ -- \eqref{eq:FDefs};  $V_{\T{LUE}}$ -- \eqref{eq:VLUE}; $\Xi$ -- \eqref{eq:XiDefs}; $\psi_n^{(\alpha)}$ -- \eqref{info:weights}; $\Phi_n^{(\alpha)}$ -- \eqref{eq:PhiDefs}; $\kappa$ -- \eqref{eq:kappaDefs}.
  \end{enumerate}

\subsection{Acknowledgements}
The author would like to thank his advisors Benjamin Landon, Jeremy Quastel, and Philippe Sosoe for their many hours of interesting discussions, patient advising, and insightful comments. This work would not be possible without their support. This work was supported by funding from an NSERC USRA.

\section{Convergence of variance (proof of Theorem \ref{thm:VarConvergence})} \label{sec:Thm1Proof}
\subsection{Relevant asymptotics and background}
We begin by defining the polynomials of interest. The normalized Laguerre polynomials are:
  \begin{equation}\label{info:weights}
    \psi_n^{(\alpha)}(x) := \Big(\frac{n!}{\Gamma(n+\alpha + 1)}\Big)^{\frac{1}{2}}e^{-x/2}x^{\alpha/2} L_n^\alpha(x), \T{ with } \int e^{-x}x^\alpha L_n^{(\alpha)}L_m^{(\alpha)}  = \frac{\Gamma(n+\alpha + 1)}{n!}\delta_{nm}.
  \end{equation}
We will be using this definition of $\psi_n^{(\alpha)}$ for the rest of the paper. Using symmetry, we can write the variance of the LSS of the LUE \eqref{eq:varKernalRep} as:
  \begin{equation*} \label{eq:varLinStat}
  \begin{split}
   \T{Var}(\mathcal{N}_n^\circ[f]) & = \frac{1}{2}\int_0^\infty \int_0^\infty F(x, y) n(n+\alpha) [\psi_{n-1}^{(\alpha)}(nx)\psi_{n}^{(\alpha)}(ny) - \psi_{n}^{(\alpha)}(nx)\psi_{n-1}^{(\alpha)}(ny)]^2 dx dy \\
    & = \int_0^\infty \int_0^\infty F(x, y) \Phi_n^{(\alpha)}(x, y) dx dy,
  \end{split}
  \end{equation*}
where $F$ is given by \eqref{eq:FDefs} and we define:
  \begin{equation} \label{eq:PhiDefs}
   \Phi_n^{(\alpha)}(x, y) := n(n+\alpha) [\psi_{n-1}^{(\alpha)}(nx)^2 \psi_{n}^{(\alpha)}(ny)^2 -\psi_{n}^{(\alpha)}(nx)\psi_{n-1}^{(\alpha)}(nx)\psi_{n}^{(\alpha)}(ny)\psi_{n-1}^{(\alpha)}(ny)].
  \end{equation}
We introduce several asymptotics for the Laguerre polynomials which we will use later. In particular we will need a bulk asymptotic as well as universal asymptotics for the hard and soft edges, all of which will be given in a simplified form. 

We introduce the convention of using $\tilde{n}$ to denote anything of form $\tilde{n} = n-k$ where $k$ is an integer not dependent on $n$. Recall the definition of $\psi_{\nn}^{(\alpha)}$ given by \eqref{info:weights}. Using $\sin^2\phi = \frac{1-\cos 2\phi}{2}$ and the asymptotics for the Laguerre polynomials from (\cite{Szego1975}, 8.22.9), we get the bulk asymptotic:
  \begin{equation} \label{eq:PRBulkSquare}
  \begin{split}
    n\psi_{\nn}^{(\alpha)}(nx)^2 & = \frac{(1-O(n^{-1}))}{2\pi\sqrt{x} \sin \phi} \Big[1-\cos\Big(\Big(2\nn + \alpha + 1\Big)(\sin2\phi - 2\phi) + \frac{3\pi}{2}\Big) + O((\nn x)^{-\frac{1}{2}})\Big], \\ &\T{where } nx = (4\nn+2\alpha + 2)\cos^2 \phi, \qquad \delta \le \phi \le \pi/2-\delta \nn^{-\frac{1}{2}}/4, \qquad \delta > 0.
  \end{split}
  \end{equation}
This asymptotic holds, given sufficiently large $n$, for $\delta < x < 4-\delta$. Near the soft edge, after simplification using Stirling's formula, we have the uniform asymptotic for $2 \le x$ (\cite{Muckenhoupt1970}, 1.2):
  \begin{equation} \label{eq:UnifSoftEdge}
   \sqrt{n}\psi_{\nn}^{(\alpha)}(nx) = (-1)^{\nn}(1+O(n^{-1}))\frac{(4\nn+2\alpha+2)^{\frac{1}{6}}\abs{\zeta(t)}^{\frac{1}{4}}}{\abs{x(1-t)}^{\frac{1}{4}}}\Big[\T{Ai}(-(4\nn+2\alpha+2)^{\frac{2}{3}}\zeta(t)) + O(E(x))\Big],
  \end{equation}
where $t = nx/(4\nn+2\alpha + 2)$,
  \begin{equation} \label{eq:zetaDefs}
    \zeta(t) := \begin{cases}
        \Big[\frac{3}{4}(\arccos \sqrt{t} - (t-t^2)^\frac{1}{2})\Big]^{\frac{2}{3}}, & 0 < t \le 1 \\
        -\Big[\frac{3}{4}((t^2-t)^\frac{1}{2} - \arccosh \sqrt{t})\Big]^{\frac{2}{3}}, & 1 < t.
  \end{cases}
  \end{equation}
and the error $E(x)$ is given in terms of the Airy functions of first and second kinds, Ai and Bi,
  \begin{equation} \label{eq:errorDefs}
  E(x) := \begin{cases}
    \frac{\tilde{\T{Ai}}(-(4\nn+2\alpha+2)^{\frac{2}{3}}\zeta(t))}{nx}, & \zeta(t) > 0\\
    \frac{\T{Ai}(-(4\nn+2\alpha+2)^{\frac{2}{3}}\zeta(t))}{nx}, & \zeta(t) \le 0,
  \end{cases}
  \end{equation}
where $\tilde{\T{Ai}}(x) = \sqrt{\T{Ai}^2(x) + \T{Bi}^2(x)}$. It follows that the error $E(x)$ is $O(n^{-1})$ for $x$ in a bounded region. Also $\zeta(t) > 0$ iff $t < 1$. Near the hard edge, once again after simplification using Stirling's formula, we have the uniform asymptotic for $0 < x \le 2$ (\cite{Muckenhoupt1970}, 1.1):
  \begin{equation} \label{eq:UnifHardEdge}
    \sqrt{n}\psi_{\nn}^{(\alpha)}(nx) = (1 + O(n^{-1}))(4\nn+2\alpha+2)^{\frac{1}{2}}\frac{(\eta(t))^{\frac{1}{2}}}{2x^{\frac{1}{4}}} (1-t)^{-\frac{1}{4}}[J_\alpha((4\nn+2\alpha+2)\eta(t)) + x^\frac{1}{2}O(n^{-1})],
  \end{equation}
where $t = nx/(4\nn+2\alpha + 2)$, $\eta(t) := \frac{1}{2}(t-t^2)^{\frac{1}{2}} + \frac{1}{2}\arcsin \sqrt{t}$, and $J_\alpha$ is the Bessel function of the first kind.

\subsection{Setup}
To prove Theorem \ref{thm:VarConvergence}, we will cut up the domain of integration in $\T{Var}(\mathcal{N}_n^\circ[f])$, namely $[0, \infty)^2$, into several smaller overlapping pieces, each of which in turn will need to be dealt with in even smaller sections:
  \begin{equation*}
    \Omega_1^{\epsilon_1} = [0, 4+\epsilon_1)^2 ,\ \ \Omega_2^{\epsilon_2} = [4+\epsilon_2, \infty)^2,\ \ \Omega_3 = [0, \infty)^2 \setminus (\Omega_1^{\epsilon_1} \cup \Omega_2^{\epsilon_2}),
  \end{equation*}
where $\epsilon_i > 0$ is small and $\epsilon_2$ is much less than $\epsilon_1$. For both $\Omega_1^{\epsilon_1}$ and $\Omega_2^{\epsilon_2}$, we will need to separately deal with the part where $x$ and $y$ are very close as well as when they are both close to the hard or soft edge, as that is where $\frac{1}{(x-y)^2}$ and\textbackslash or $\Phi_n^{(\alpha)}$ blows up. It will be convenient to prove a result concerning $\Omega_3$ first. 

\subsection{Convergence in \texorpdfstring{$\Omega_3$}{Omega3}} \label{sec:LInfinityProof}
In $\Omega_3$ we note that $\abs{x - y}$ is bounded away from $0$ which implies $F(x, y)$ is $L^\infty$. This motivates the following (more general) Proposition:
\begin{prop} \label{prop:Double}
For $f(x, y)\in L^\infty([0, \infty)^2)$,
  \begin{equation*}
    \lim_{n\to\infty} \int_0^\infty\int_0^\infty f(x, y) \Phi_n^{(\alpha)}(x, y) dx dy = \int_0^4\int_0^4 f(x, y)\Xi(x, y) dx dy,
  \end{equation*}
  where $\Phi_n^{(\alpha)}(x, y)$ and $\Xi(x, y)$ are defined in \eqref{eq:PhiDefs}, \eqref{eq:XiDefs}.
\end{prop}
The rest of this section is devoted to proving Proposition \ref{prop:Double}. We will deal with $\Phi_n^{(\alpha)}(x,y)$ term by term. Namely we will work with the square term $\psi_{n-1}^{(\alpha)}(nx)^2 \psi_{n}^{(\alpha)}(ny)^2$ and show it converges, and then likewise the cross term $\psi_{n}^{(\alpha)}(nx)\psi_{n-1}^{(\alpha)}(nx)\psi_{n}^{(\alpha)}(ny)\psi_{n-1}^{(\alpha)}(ny)$. For both terms, we will first show that in the bulk, that is $f\in L^\infty([\delta, 4-\delta]^2)$, the convergence holds. Then we will extend it to $f\in L^\infty([0,\infty)^2)$ by showing that everything outside the bulk only has negligible contribution. The method of proof here was inspired by Kopel's preprint \cite{Kopel2015}.
\subsubsection{Square term} \label{sec:SquareTerm}
We start with the simplest case and slowly generalize. 
\begin{lemma} \label{lem:SquareBulk}For $f\in L^\infty([\delta, 4-\delta])$ for small $\delta > 0$, we have, 
  \begin{equation*}
    \lim_{n\to\infty} n\int_{\delta}^{4-\delta} f(x) \psi^{(\alpha)}_{\nn}(nx)^2 dx = \int_{\delta}^{4-\delta} \frac{2f(x)}{\kappa(x)} dx.
  \end{equation*}
  where we define,
    \begin{equation} \label{eq:kappaDefs}
     \kappa(x) := 2\pi\sqrt{4-(x-2)^2} = 2\pi\sqrt{4x-x^2}
    \end{equation}
\end{lemma}
\begin{proof} 
Let $\delta > 0$ be given. Starting with the asymptotic \eqref{eq:PRBulkSquare} and taking the integral against a test function $f(x) \in L^\infty([\delta, 4-\delta])$:
  \begin{equation} \label{eq:IntegralSquareBulk}
  \begin{split}
   n\int_{\delta}^{4-\delta} f(x) \psi_{\nn}^{(\alpha)}(nx)^2 dx & = (1-O(n^{-1}))\bigg[ (1+O(\nn^{-\frac{1}{2}}))\int_{\delta}^{4-\delta} \frac{f(x)}{2\pi \sqrt{x} \sin\phi} dx \\
    & \qquad \qquad - \int_{\delta}^{4-\delta} f(x) \frac{\sin\Big(\Big(2\nn + \alpha + 1\Big)(\sin2\phi - 2\phi)\Big)}{2\pi\sqrt{x} \sin \phi}dx \bigg].
  \end{split}
  \end{equation}
As $2\sqrt{x}\pi\sin\phi$ is bounded away from $0$ for $x\in [\delta, 4-\delta]$ the integrands on right hand side of \eqref{eq:IntegralSquareBulk} are all bounded and all the lower order terms disappear in the limit. The oscillatory term on the second line goes away as a consequence of the Riemann-Lebesgue Lemma. To see this let $a(\phi) = \sin 2\phi - 2\phi$. Since $x = \frac{4\nn+2\alpha + 2}{n}\cos^2\phi \in [\delta, 4-\delta]$, we see that $1-\cos 2\phi$ is bounded away from $0$. Thus $\abs{\frac{(4\nn+2\alpha+2) \cos\phi}{2n \pi\sqrt{x} (1-\cos 2\phi)}} = g(a(\phi)) + O(n^{-1})$ where $g$ does not depend on $n$ and hence:
  \begin{equation*}
  \begin{split}
  \lim_{n\to\infty} \int_{\delta}^{4-\delta} f(x) & \frac{\sin\Big(\Big(2\nn + \alpha + 1\Big)(\sin2\phi - 2\phi)\Big)}{2\pi\sqrt{x} \sin \phi}dx \\
    & = \lim_{n\to\infty} \int_{\delta}^{4-\delta} f(x) \frac{(4\nn+2\alpha+2) \cos\phi}{2 n \pi\sqrt{x} (1-\cos 2\phi)}\sin\Big(\Big(2\nn + \alpha + 1\Big)a(\phi)\Big) d[a(\phi)] \\
    & = 0.
  \end{split}
  \end{equation*}
As $x = 4\cos^2\phi + O(\nn^{-1})$, the $\sin \phi$ in the denominator of \eqref{eq:IntegralSquareBulk} becomes $\sqrt{1-x/4}$ in the limit $n\to\infty$. Hence, 
  \begin{equation*}
    \lim_{n\to\infty}n\int_{\delta}^{4-\delta} f(x) \psi_{\nn}^{(\alpha)}(nx)^2 dx = \int_{\delta}^{4-\delta} \frac{f(x)}{\pi \sqrt{4x-x^2}} dx.
  \end{equation*}
\end{proof}
As we wish to extend this result to functions $f(x,y)\in L^\infty([0,\infty)^2)$, we need an estimate on how uniform the convergence is. While the result is a fairly trivial application of Egoroff's theorem, it is important enough to warrant its own Lemma.
\begin{lemma} \label{lem:SquareBulkError}
Let $f(x, y)\in L^\infty([\delta, 4-\delta]^2)$ and $\epsilon > 0$ be given. Define $f_y(x) = f(x, y)$. Then there exists $N$ such that for all $n > N$ and $y$ except possibly for $y$ in some (fixed) set of measure-$\epsilon$:
  \begin{equation*}
    \abs*{n\int_{\delta}^{4-\delta} f_y(x) \psi^{(\alpha)}_{\nn}(nx)^2 dx - \int_{\delta}^{4-\delta} \frac{2f_y(x)}{\kappa(x)} dx} < \epsilon,
  \end{equation*}
 where $\kappa$ is defined \eqref{eq:kappaDefs}.
\end{lemma}
\begin{proof}
This follows from Egoroff's theorem applied to the function 
\[y \mapsto n\int_{\delta}^{4-\delta} f_y(x) \psi^{(\alpha)}_{\nn}(nx)^2 dx.\]
\end{proof}
This allows us to prove the bulk convergence result for the square term of $\Phi_n^{(\alpha)}(x,y)$ as the convergence in Lemma \ref{lem:SquareBulk} is uniform nearly everywhere. 
\begin{lemma} \label{lem:DoubleSquareBulk}
  Let $f(x,y)\in L^\infty([\delta, 4-\delta]^2)$. Then for small $\delta > 0$ we have:
  \begin{equation*}
  \begin{split}
    \lim_{n\to\infty} \int_\delta^{4-\delta} \int_\delta^{4-\delta} n(n+\alpha)f(x, y)\psi_n^{(\alpha)}(nx)^2 & \psi_{n-1}^{(\alpha)}(ny)^2 dxdy = \int_\delta^{4-\delta}\int_\delta^{4-\delta} \frac{4f(x, y)}{\kappa(x)\kappa(y)} dxdy,
  \end{split}
  \end{equation*}
  where $\kappa$ is defined \eqref{eq:kappaDefs}.
\end{lemma} 
\begin{proof}
Let $f_y(x) = f(x, y)$. From Lemma \ref{lem:SquareBulkError}, there exists $N$ so that for all $n > N$ and all $y$ except possibly for some set $\Sigma_\epsilon$ of measure $\frac{\epsilon \pi\delta}{8\norm{f}_\infty}$ the following holds:
  \[\abs*{n\int_\delta^{4-\delta} f_y(x) \psi_n^{(\alpha)}(nx)^2 dx - \int_\delta^{4-\delta} \frac{2f_y(x)}{\kappa(x)} dx} < \min\Big(\frac{\epsilon}{4\norm{(n+\alpha)\psi_n^{(\alpha)}(ny)^2}_1}, \frac{\epsilon \pi\delta}{8\norm{f}_\infty}\Big).\]
We denote $\Sigma_\epsilon^c = [\delta, 4-\delta]\setminus \Sigma_\epsilon$. Then we get for $n > N$,
  \begin{equation}\label{eq:DoubleNiceBulkSquare}
  \begin{split}
    & \ \ \ \ \abs*{\int_{\Sigma_\epsilon^c} (n+\alpha)\psi_{n-1}^{(\alpha)}(ny)^2 \int_\delta^{4-\delta} \Big(nf(x, y)\psi_n^{(\alpha)}(nx)^2  -\frac{2f(x, y)}{\kappa(x)}\Big) dxdy} \\
    & < \int_{\Sigma_\epsilon^c} \frac{\epsilon (n+\alpha)\psi_{n-1}^{(\alpha)}(ny)^2 }{4\norm{(n+\alpha)\psi_n^{(\alpha)}(ny)^2}_1} dy = \frac{\epsilon}{4}.
  \end{split}
  \end{equation}
At the same time, since $f$ is $L^\infty$, H\"{o}lder's inequality with the orthogonality of $\psi_n^{(\alpha)}$ \eqref{info:weights} tells us that for almost every $y$, 
  \begin{equation*}
  \begin{split}
    \abs*{\int_\delta^{4-\delta} nf_y(x) \psi_n^{(\alpha)}(nx)^2 -\frac{2f_y(x)}{\kappa(x)} dx} & \le \norm{f}_\infty \int_\delta^{4-\delta} \Big(n\psi_n^{(\alpha)}(nx)^2 + \frac{2}{\kappa(x)}\Big) dx \\
      & \le 2\norm{f}_\infty.
  \end{split}
  \end{equation*}
Hence Lemma \ref{lem:SquareBulk} (applied for the inequality on the third line) tells us for $n > N$,
  \begin{equation}\label{eq:DoubleMeanBulkSquare}
  \begin{split}
    & \ \ \ \ \abs*{\int_{\Sigma_\epsilon} (n+\alpha)\psi_{n-1}^{(\alpha)}(4ny)^2 \int_\delta^{4-\delta}\Big(nf(x, y)\psi_n^{(\alpha)}(nx)^2 -\frac{2f(x, y)}{\kappa(x)}\Big) dxdy} \\
      & \le 2\norm{f}_\infty (n+\alpha)\int_{\Sigma_\epsilon}\psi_{n-1}^{(\alpha)}(ny)^2 dx \\
      & \le 2\norm{f}_\infty \abs*{\int_{\Sigma_\epsilon}\frac{2}{\kappa(y)}dx + \frac{\epsilon}{8\norm{f}_\infty}} \\
      & \le 2\norm{f}_\infty \abs*{\frac{\epsilon \pi\delta}{8\norm{f}_\infty}\frac{1}{\pi\delta} + \frac{\epsilon}{8\norm{f}_\infty}} = \frac{\epsilon}{4}.
  \end{split}
  \end{equation}
Adding \eqref{eq:DoubleNiceBulkSquare} and \eqref{eq:DoubleMeanBulkSquare}, we get that given $\epsilon > 0$ for $n > N$,
  \begin{equation} \label{eq:01DoubleBulkSquare1}
  \begin{split}
    & \abs*{\int_\delta^{4-\delta}(n+\alpha)\psi_n^{(\alpha)}(4ny)^2\int_\delta^{4-\delta} nf(x, y)\psi_n^{(\alpha)}(nx)^2 - \frac{2f(x, y)}{\kappa(x)} dxdy} < \frac{\epsilon}{2}.
  \end{split} 
  \end{equation}
As $\int_\delta^{4-\delta} \frac{2f(x, y)}{\kappa(x)} dx$ is uniformly bounded for almost every $y$ and hence $L^\infty$, we can apply Lemma \ref{lem:SquareBulk} to see that for $n > N$:
  \begin{equation}\label{eq:01DoubleBulkSquare2}
    \abs*{\int_\delta^{4-\delta}(n+\alpha)\psi_n^{(\alpha)}(ny)^2\int_\delta^{4-\delta} \frac{2f(x, y)}{\kappa(x)} dxdy - \int_\delta^{4-\delta}\int_\delta^{4-\delta} \frac{4f(x, y)}{\kappa(x)\kappa(y)} dxdy} < \frac{\epsilon}{2}.
  \end{equation}
Adding \eqref{eq:01DoubleBulkSquare1} and \eqref{eq:01DoubleBulkSquare2}, we see that there exists $N$ so that for $n > N$:
  \begin{equation*}
    \bigg\lvert\int_\delta^{4-\delta} \int_\delta^{4-\delta}n(n+\alpha)f(x, y)\psi_n^{(\alpha)}(nx)^2 \psi_{n-1}^{(\alpha)}(ny)^2 dxdy  - \int_\delta^{4-\delta}\int_\delta^{4-\delta} \frac{4f(x, y)}{\kappa(x)\kappa(y)} dxdy\bigg\rvert < \epsilon.
  \end{equation*}
\end{proof}

We illustrate the key idea to extend the results to $f\in L^\infty([0,\infty)^2)$. Notice that
  \[\int_0^4 \frac{2}{\kappa(x)}dx = \int_0^4 \frac{2}{2\pi \sqrt{4-(x-2)^2}}dx = 1.\]
In particular, because $\psi_{\nn}^{(\alpha)}(x)$ are orthonormal by construction \eqref{info:weights}, $n\int_0^\infty \psi_{\nn}^{(\alpha)}(nx)^2 dx = 1$ for all $n$. So given $\nu > 0$, by the continuity of integrals we can find sufficiently small $\delta > 0$ such that:
  \begin{equation*} \label{eq:SquareEdgeBound}
  \begin{split}
   \lim_{n\to\infty} n\int_{[0, \infty)\setminus [\delta, 4-\delta]} \psi_{\nn}^{(\alpha)}(nx)^2 dx & = \lim_{n\to\infty} \bigg[ n\int_0^\infty \psi_{\nn}^{(\alpha)}(nx)^2 dx - n\int_{\delta}^{4-\delta} \psi_{\nn}^{(\alpha)}(nx)^2 dx\bigg] \\
    & = 1 - \int_\delta^{4-\delta} \frac{2}{\kappa(x)} dx  \\
    & \le \nu,
  \end{split}
  \end{equation*}
where we applied Lemma \ref{lem:SquareBulk} for the second equality. This idea applies for the cross term of $\Phi_n^{(\alpha)}(x, y)$ as well. We will write the details out for the square term.

\begin{lemma} \label{lem:DoubleSquare}
For $f(x,y)\in L^\infty([0,\infty))^2$, we have, 
  \begin{equation*}
    \lim_{n\to\infty} \int_0^\infty \int_0^\infty n(n+\alpha)f(x, y)\psi_n^{(\alpha)}(nx)^2 \psi_{n-1}^{(\alpha)}(ny)^2 dxdy = \int_0^4 \int_0^4 \frac{4f(x, y)}{\kappa(x)\kappa(y)} dxdy,
  \end{equation*} 
where $\kappa(x)$ is defined \eqref{eq:kappaDefs}.
\end{lemma}
\begin{proof}
Notice that
  \[\int_0^4\int_0^4 \frac{4}{\kappa(x)\kappa(y)}dx = 1.\]
This allows us to show the contribution outside the bulk is small. Let $f(x, y)\in L^\infty([0,\infty)^2)$ and $\epsilon > 0$ be given. Because $\psi_n^{(\alpha)}(x)$ are orthonormal \eqref{info:weights},
  \begin{align*}
    \lim_{n\to\infty}\int_0^\infty\int_0^\infty  n(n+\alpha)\psi_n^{(\alpha)}(nx)^2\psi_{n-1}^{(\alpha)}(ny)^2 dxdy  = 1.
  \end{align*}
Let $\Sigma_\delta = [0, \infty)^2\setminus [\delta, 4-\delta]^2$. Then applying Lemma \ref{lem:DoubleSquareBulk} and using the continuity of integrals, we can choose $\delta > 0$ to be sufficiently small so that:
  \begin{equation} \label{eq:DoubleEdgeBound}
   \lim_{n\to\infty} \iint_{\Sigma_\delta} n(n+\alpha)\psi_n^{(\alpha)}(nx)^2\psi_{n-1}^{(\alpha)}(ny)^2 dxdy = 1 - \int_\delta^{4-\delta} \frac{4}{\kappa(x)\kappa(y)} dx \le \frac{\epsilon}{6\norm{f}_\infty+1}.
  \end{equation}
Hence we see that,
  \begin{equation*}
  \begin{split}
    & \ \ \ \ \abs*{\int_0^4\int_0^4 \frac{4f(x,y)}{\kappa(x)\kappa(y)} dxdy - \int_0^\infty\int_0^\infty n(n+\alpha)f(x,y) \psi_n^{(\alpha)}(nx)^2\psi_{n-1}^{(\alpha)}(ny)^2 dxdy} \\
      &  \le \abs*{\int_{0}^4\int_0^4 \frac{4f(x,y)}{\kappa(x)\kappa(y)} dxdy - \int_{\delta}^{4-\delta} \int_{\delta}^{4-\delta} \frac{4f(x,y)}{\kappa(x)\kappa(y)} dxdy} \\
    & \qquad + \abs*{\int_{\delta}^{4-\delta}\int_{\delta}^{4-\delta} \frac{4f(x,y)}{\kappa(x)\kappa(y)} dxdy - \int_{\delta}^{4-\delta}\int_{\delta}^{4-\delta} n(n+\alpha) f(x,y) \psi_n^{(\alpha)}(nx)^2 \psi_{n-1}^{(\alpha)}(ny)^2dxdy}  \\
    & \qquad + \abs*{\iint_{\Sigma_\delta} n(n+\alpha) f(x,y) \psi_n^{(\alpha)}(nx)^2 \psi_{n-1}^{(\alpha)}(ny)^2dxdy}.
  \end{split}
  \end{equation*}
Because $f$ is $L^\infty$, the first term on the right hand side is less than $\frac{\epsilon}{3}$ for sufficiently small $\delta > 0$. H\"{o}lder's inequality with \eqref{eq:DoubleEdgeBound} tells us for sufficiently large $n > N$, with small $\delta$, the third term is also bounded by $\frac{\epsilon}{3}$. After fixing $\delta$, it follows from Lemma \ref{lem:DoubleSquareBulk} that there exists $N$ so that $n > N$ implies the second term is less than $\frac{\epsilon}{3}$. Hence there exists $N$ such that $n > N$ implies:
  \begin{equation*}
    \abs*{\int_0^4\int_0^4 \frac{4f(x,y)}{\kappa(x)\kappa(y)} dxdy - \int_0^\infty\int_0^\infty n(n+\alpha)f(x,y) \psi_n^{(\alpha)}(nx)^2\psi_{n-1}^{(\alpha)}(ny)^2 dxdy} < \epsilon.
  \end{equation*}
\end{proof}

\subsubsection{Cross term}
The cross term does not behave quite as nicely with our asymptotics. However, we wish to do something similar to how we showed convergence for the square term. We will show the bulk case for the cross term $\psi_n^{(\alpha)}\psi_{n-1}^{(\alpha)}$ by first proving it for smooth functions and then approximating. We use the recurrence relations for Laguerre polynomials to express the cross term as the sum of the square term with some error, which disappears in the limit after integration by parts.
\begin{lemma} \label{lem:BulkSmoothCross} For compactly supported $f\in C^\infty((\delta,4-\delta))$ for small $\delta > 0$,
  \[\lim_{n\to\infty} n\int_\delta^{4-\delta} f(x) \psi_n^{(\alpha)}(nx)\psi_{n-1}^{(\alpha)}(nx) = \int_\delta^{4-\delta} \frac{(2-x)f(x)}{\kappa(x)} dx.\]
where $\kappa(x)$ is defined \eqref{eq:kappaDefs}.
\end{lemma}
\begin{proof}
Let $\delta > 0$ be given. From the recurrence relation for Laguerre polynomials (\cite{Szego1975}, 5.1.14), we see the recurrence relation for $\psi_n^{(\alpha)}$ is given by,
  \begin{equation} \label{eq:psiDerivativeRecurrence}
    \frac{d}{dx}\psi_n^{(\alpha)}(x) = \frac{2n+\alpha - x}{2x} \psi_n^\alpha(x) - \frac{\sqrt{n(n+\alpha)}}{x}\psi_{n-1}^{(\alpha)}(x),
  \end{equation}
From the above recurrence relation we get the equality:
  \begin{equation} \label{eq:crossSimplify}
  \begin{split}
   n\psi_n^{(\alpha)}(nx)\psi_{n-1}^{(\alpha)}(nx) & = n\psi_n^{(\alpha)}(nx)\Big[\frac{2n - nx +\alpha}{2\sqrt{n(n+\alpha)}}\psi_n^{(\alpha)}(nx) - \frac{nx}{\sqrt{n(n+\alpha)}}\psi_n^{(\alpha)\prime}(nx)\Big].
  \end{split}
  \end{equation}
Integrating $\psi_n^{(\alpha)}(nx)\psi_n^{(\alpha)\prime}(nx)$ against a compactly supported smooth test function $f(x)\in C^\infty((\delta, 4-\delta))$ and using integration by parts:
  \begin{equation} \label{eq:crossInt}
   n\int \frac{nxf(x)}{\sqrt{n(n+\alpha)}}\psi_n^{(\alpha)}(nx)\psi_n^{(\alpha)\prime}(nx)dx = - \frac{n}{2}\int \Big(\frac{f(x)\psi_n^{(\alpha)}(nx)^2}{\sqrt{n(n+\alpha)}} + \frac{xf'(x)\psi_n^{(\alpha)}(nx)^2}{\sqrt{n(n+\alpha)}}\Big)dx.
  \end{equation}
As the derivative $f'$ is bounded in $[\delta, 4-\delta]$, in the limit $n\to\infty$, the right hand side of equation \eqref{eq:crossInt} vanishes because of Lemma \ref{lem:SquareBulk}. So applying Lemma \ref{lem:SquareBulk} to \eqref{eq:crossSimplify}, we get the desired conclusion:
  \begin{equation*}
  \begin{split}
    \lim_{n\to\infty} n\int_\delta^{4-\delta} f(x) \psi_n^{(\alpha)}(nx)\psi_{n-1}^{(\alpha)}(nx) dx
      & = \int_\delta^{4-\delta} \frac{(2-x)f(x)}{\kappa(x)} dx.
  \end{split}
  \end{equation*}
\end{proof}
We extend the regularity condition of the above lemma to $L^\infty$ by approximating $L^p$ functions with compact support by smooth mollifications. Note that this works only because we are on a domain with finite measure. 
\begin{lemma} \label{lem:CrossBulk}For $f\in L^\infty([\delta, 4-\delta])$ for small $\delta > 0$,
  \begin{equation*} \lim_{n\to\infty} n\int_{\delta}^{4-\delta} f(x)\psi_n^{(\alpha)}(nx)\psi_{n-1}^{(\alpha)}(nx) dx = \int_{\delta}^{4-\delta} \frac{(2-x)f(x)}{\kappa(x)} dx.
  \end{equation*}
where $\kappa(x)$ is defined \eqref{eq:kappaDefs}.
\end{lemma}
\begin{proof}
Let $\epsilon > 0$ and $\delta > 0$ be given. Note that because $[0,4]$ has finite measure, $L^\infty([0, 4])\subset L^p([0, 4])$ for all $1\le p\le\infty$. So by approximating using mollifiers, we can find a sequence of functions that converges to $f$ in $L^p$ for all $1\le p<\infty$. 

In particular, take a sequence of smooth functions $f_m$ that converges to $f$ as $m\to\infty$ in $L^p$, where $1 \le p < \infty$. We can require that $f_m$ are all supported in $[\frac{\delta}{2}, 4-\frac{\delta}{2}]$. From the asymptotic \eqref{eq:PRBulkSquare}, we see that $n\psi_n^{(\alpha)}(nx)^2$ and $n\psi_{n-1}^{(\alpha)}(nx)^2$ are both uniformly bounded by some constant $C$ for $x\in[\frac{\delta}{2}, 4-\frac{\delta}{2}]$ and all $n > N$. In particular, for sufficiently large $M_1$, $m \ge M_1$ implies: 
  \begin{equation} \label{eq:CrossMLim1}
  \begin{split}
   & \ \ \ \ \abs*{n\int_{\frac{\delta}{2}}^{4-\frac{\delta}{2}} (f(x) - f_m(x))\psi_n^{(\alpha)}(nx)\psi_{n-1}^{(\alpha)}(nx) dx} \\
   & \le \int_{\frac{\delta}{2}}^{4-\frac{\delta}{2}} \abs{f(x) - f_m(x)}\abs{n\psi_n^{(\alpha)}(nx)^2 + n\psi_{n-1}^{(\alpha)}(nx)^2} dx \\
   & \le \norm{n\psi_n^{(\alpha)}(nx)^2 + n\psi_{n-1}^{(\alpha)}(nx)^2}_\infty \int_{\frac{\delta}{2}}^{4-\frac{\delta}{2}} \abs{f(x) - f_m(x)} dx \\
   & \le 2C \int_{\frac{\delta}{2}}^{4-\frac{\delta}{2}} \abs{f(x) - f_m(x)} dx \le \frac{\epsilon}{3}.
  \end{split}
  \end{equation}
By Lemma \ref{lem:BulkSmoothCross}, we have:
  \begin{equation} \label{eq:CrossNLim}
    \lim_{n\to\infty} n\int_{\frac{\delta}{2}}^{4-\frac{\delta}{2}} f_m \psi_n^{(\alpha)}(nx)\psi_{n-1}^{(\alpha)}(nx)dx = \int_{\frac{\delta}{2}}^{4-\frac{\delta}{2}} \frac{(2-x)f_m(x)}{\kappa(x)} dx.
  \end{equation}
Because $\kappa(x)$ is bounded away from $0$ on $[\frac{\delta}{2}, 4-\frac{\delta}{2}]$ and $f$ is supported in $[\delta, 4-\delta]$ we also have:
  \begin{equation} \label{eq:CrossMLim2}
    \lim_{m\to\infty} \int_{\frac{\delta}{2}}^{4-\frac{\delta}{2}} \frac{(2-x)f_m(x)}{\kappa(x)} dx = \int_{\delta}^{4-\delta}  \frac{(2-x)f(x)}{\kappa(x)}dx.
  \end{equation}
Hence for $n > N$:
  \begin{equation*} 
  \begin{split}
     & \ \ \ \ \abs*{n\int_{\delta}^{4-\delta} f(x) \psi_n^{(\alpha)}(nx)\psi_{n-1}^{(\alpha)}(nx)dx - \int_{\delta}^{4-\delta}  \frac{(2-x)f(x)}{\kappa(x)}dx} \\
     & \le \abs*{n\int_{\delta}^{4-\delta} f(x) \psi_n^{(\alpha)}(nx)\psi_{n-1}^{(\alpha)}(nx)dx - n\int_{\frac{\delta}{2}}^{4-\frac{\delta}{2}} f_m(x) \psi_n^{(\alpha)}(nx)\psi_{n-1}^{(\alpha)}(nx)dx} \\
     & \qquad + \abs*{n\int_{\frac{\delta}{2}}^{4-\frac{\delta}{2}} f_m(x) \psi_n^{(\alpha)}(nx)\psi_{n-1}^{(\alpha)}(nx)dx -  \int_{\frac{\delta}{2}}^{4-\frac{\delta}{2}} \frac{(2-x)f_m(x)}{\kappa(x)} dx} \\
     & \qquad+ \abs*{\int_{\frac{\delta}{2}}^{4-\frac{\delta}{2}} \frac{(2-x)f_m(x)}{\kappa(x)} dx - \int_{\delta}^{4-\delta}  \frac{(2-x)f(x)}{\kappa(x)}dx}.
  \end{split}
  \end{equation*}
The first term of the right hand side is less than $\frac{\epsilon}{3}$ when $m\ge M_1$ because of \eqref{eq:CrossMLim1}. Similarly, it follows from \eqref{eq:CrossMLim2} that for $m \ge M_2$, the third term is less than $\frac{\epsilon}{3}$. Finally, fixing $m = \max(M_1, M_2)$, \eqref{eq:CrossNLim} tells us there exists $N$ such that $n > N$ bounds the second term by $\frac{\epsilon}{3}$. Hence for all $n > N$, 
  \begin{equation*}
    \abs*{n\int_{\delta}^{4-\delta} f(x) \psi_n^{(\alpha)}(nx)\psi_{n-1}^{(\alpha)}(nx)dx - \int_{\delta}^{4-\delta}  \frac{(2-x)f(x)}{\kappa(x)}dx} < \epsilon.
  \end{equation*}
\end{proof}
To extend this to the cross term of $\Phi_n(x,y)$, i.e. 
  \[n(n+\alpha)\psi^{(\alpha)}_n(nx)\psi_{n-1}^{(\alpha)}(nx)\psi^{(\alpha)}_n(ny)\psi_{n-1}^{(\alpha)}(ny),\]
we repeat the same steps as the square term, found in Section \ref{sec:SquareTerm}. Namely, we use Egoroff's theorem to arrive at the equivalent of Lemma \ref{lem:SquareBulkError} for the cross term. The method of proving the corresponding Lemma of \ref{lem:DoubleSquareBulk} and \ref{lem:DoubleSquare} for the cross term is identical as before, except we substitute the inequality $\abs{2\psi_{n}^{(\alpha)}(nx)\psi_{n-1}^{(\alpha)}(nx)} < \psi_{n}^{(\alpha)}(nx)^2 + \psi_{n-1}^{(\alpha)}(nx)^2$ where necessary. Of course, it would use Lemma \ref{lem:CrossBulk} instead of Lemma \ref{lem:SquareBulk}. So we get:

\begin{lemma} \label{lem:DoubleCross}
For $f(x,y)\in L^\infty([0,\infty)^2)$, we have, 
  \begin{equation*}
  \begin{split}
    & \ \ \ \ \lim_{n\to\infty} \int_0^\infty \int_0^\infty n(n+\alpha)f(x, y)\psi_{n}^{(\alpha)}(nx)\psi_{n-1}^{(\alpha)}(nx)\psi_{n}^{(\alpha)}(ny)\psi_{n-1}^{(\alpha)}(ny) dxdy \\
    & = \int_0^4 \int_0^4 \frac{(x-2)(y-2)f(x, y)}{\kappa(x)\kappa(y)} dxdy.
  \end{split}
  \end{equation*} 
where $\kappa$ is defined \eqref{eq:kappaDefs}.
\end{lemma}
\begin{proof} 
  This is essentially combining the proof for Lemma \ref{lem:DoubleSquare} with Lemma \ref{lem:CrossBulk} and some straightforward applications of Cauchy-Schwarz and AM-GM.  
\end{proof}
\noindent
It is obvious that Proposition \ref{prop:Double} immediately follows from Lemma \ref{lem:DoubleSquare} and \ref{lem:DoubleCross}.

\subsection{Convergence in \texorpdfstring{$\Omega_1^{\epsilon_1}$}{Omega1}}
The main point of this section is to prove Proposition \ref{prop:BulkVarianceConvergence}. To show that $\T{Var}(\mathcal{N}_n^\circ[f])$ converges (in the domain $\Omega_1^{\epsilon_1}$) to $V_{\T{LUE}}[f]$, we first show that it converges for $f\in L^\infty([0, 4])$ compactly supported in $[\delta, 4-\delta]$ for some $\delta > 0$. We then remove the requirement of compact support by estimating the contribution from the soft and hard edges. Note that since $F(x, y) \ge 0$, assumption \eqref{eq:Assumption1} implies that $F(x, y) \in L^1([0, 4+\epsilon]^2)$. 

\begin{lemma} \label{lem:BulkVariance}
  For $f(x)\in L^\infty([0, 4])$ compactly supported on $[\delta, 4-\delta]$ for some $\delta>0$ such that \eqref{eq:Assumption1} is satisfied,
  \begin{equation*}
    \lim_{n\to\infty} \int_0^4 \int_0^4 F(x, y)\Phi_n^{(\alpha)}(x, y) dx dy = \int_0^4\int_0^4 F(x, y)\Xi(x, y)dx dy,
  \end{equation*}
  where $F(x, y)$, $\Phi_n^{(\alpha)}(x,y)$, and $\Xi(x, y)$ are defined in \eqref{eq:FDefs}, \eqref{eq:PhiDefs}, and \eqref{eq:XiDefs} respectively.
\end{lemma}
\begin{proof}
Let $\delta > 0$ be given and $\nu > 0$ be much less than $\delta$. Let $B_x(a, \nu)$ denote a $\nu$-ball around $a$ in the $x$-domain and let $B_x(a, \nu)^c = [0,4]\setminus B_x(a, \nu)$ denote its complement. In particular, $B_x(y, \nu)$ is an $\nu$-ball around the value of $y$ in the $x$-domain. We will show that the contributions to the variance from such balls are small. Note that the asymptotic \eqref{eq:PRBulkSquare} provides the bound $\abs{n\psi_{\nn}^{(\alpha)}(nx)^2} \le C$ uniformly for $n > N$ and $x \in[\delta - 2\nu, 4-(\delta- 2\nu)]$. Hence for $n > N$ and using Cauchy-Schwarz, for $x, y \in [\delta - 2\nu, 4-(\delta- 2\nu)]^2$,
  \begin{equation} \label{eq:bulkBoundPhi}
    \abs{\Phi_n^{(\alpha)}(x, y)} \le 6C^2.
  \end{equation}
As $F(x,y) = 0$ for $(x, y) \in [0, \delta - \nu]^2$ and $(x, y)\in [4-(\delta - \nu), 4]^2$, we have that:
  \begin{align*}
   & \abs*{\int_0^4 \int_{B_x(y, \nu)} F(x,y) \Phi_n^{(\alpha)}(x, y) dxdy } \le \int_0^4 \int_{B_x(y, \nu)} F(x,y) 6C^2 dxdy.
  \end{align*}
In particular, the functions $\one_{B_x(y, \nu)} F(x, y)$ converges pointwise to $0$ as $\nu \to 0$. So by the dominated convergence theorem we see that uniformly in $n > N$,
  \begin{equation} \label{eq:UniformBoundSquaredLimit}
  \begin{split}
    & \lim_{\nu \to 0}\int_0^4 \int_{B_x(y, \nu)} F(x,y) \Phi_n^{(\alpha)}(x, y) dxdy = 0.
  \end{split}
  \end{equation}
Because of the assumption \eqref{eq:Assumption1},
  \begin{equation} \label{eq:UniformBoundSquaredLimit2}
  \begin{split}
    & \lim_{\nu \to 0}\int_0^4 \int_{B_x(y, \nu)} F(x, y)\Xi(x, y) dx dy = 0.
  \end{split}
  \end{equation}
Hence given $\epsilon > 0$, we can find $\nu_1 > 0$ such that for all $0 < \nu < \nu_1$,
  \begin{equation} \label{eq:VarBulk1}
  \begin{split}
    & \abs*{\int_0^4 \int_{B_x(y, \nu)} F(x, y)\Xi(x, y)} < \frac{\epsilon}{4},\ \ \ \ \abs*{\int_0^4 \int_{B_x(y, \nu)} F(x, y) \Phi_n^{(\alpha)}(x, y) dxdy} < \frac{\epsilon}{4}.
  \end{split}
  \end{equation}
Now in the region $\Sigma_{\nu} := [0,4]^2 \setminus \{(x, y)\ |\ x \in B_x(y, \nu)\} = \{(x, y)\in[0,4]^2| \abs{x-y} \ge \nu\}$ we have:
  \[\norm*{F(x,y)}_{L^\infty(\Sigma_{\nu})} \le 4\nu^{-2}\norm{f(x)}_{L^\infty([0,4])}^2.\]
So from Propostion \ref{prop:Double} we have that for $n > N              $,
  \begin{equation} \label{eq:VarBulk2}
    \abs*{\int_0^4 \int_{B_x(y, \nu)^c}  F(x,y) \Phi_n^{(\alpha)}(x, y) dxdy - \int_0^4 \int_{B_x(y, \nu)^c} F(x,y) \Xi(x, y)}< \frac{\epsilon}{2}.
  \end{equation}
Adding \eqref{eq:VarBulk1} and \eqref{eq:VarBulk2} gives us for $n > N$,
  \begin{equation}
   \abs*{\int_0^4 \int_0^4 F(x,y) \Phi_n^{(\alpha)}(x, y) dxdy - \int_0^4\int_0^4 F(x,y)\Xi(x, y) dxdy}< \epsilon.
  \end{equation}
\end{proof}

We will now remove the condition of compact support and extend the result to $f(x)\in L^\infty([0,4])$. When $x$ and $y$ are apart, Proposition \ref{prop:Double} is enough for convergence. The problem arises when $x$ and $y$ are close, as $F(x,y)$ becomes (possibly) unbounded, and when we are close to the soft or hard edge, as $\Phi_n^{(\alpha)}$ blows up (with $n$). To deal with this, we will need to use universal asymptotics for the Laguerre polynomials to obtain more precise bounds near the hard edge and soft edge. We prove the following two bounds, corresponding to the soft and hard edge respectively. 
\begin{lemma} \label{lem:softEdgeBound}
 There exists $N$ such that for all $n > N$ and $3 <x, y <4+\epsilon$,
  \begin{equation*}
    n^2\psi_{\nn}^{(\alpha)}(ny)^2[\psi_{n}^{(\alpha)}(nx) + \psi_{n-1}^{(\alpha)}(nx)]^2 \le C_1\frac{\abs{4-x}^{\frac{1}{2}}}{\abs{4-y}^\frac{1}{2}} + C_2,
  \end{equation*}
  where $C_1$ and $C_2$ are positive constants not depending on $n$, $x$, or $y$. 
\end{lemma}
\begin{proof}
 We start by finding a bound for $\sqrt{n}\psi_{\nn}^{(\alpha)}(nx)$ when $x$ is near the soft edge. Recall $\zeta(t)$ is defined in \eqref{eq:zetaDefs}. Note that $\frac{d}{dt} \abs{\zeta(t)}^\frac{3}{2} = \mp \frac{3}{4}\frac{\sqrt{\abs{1-t}}}{\sqrt{t}}$ where it takes the negative sign for $t < 1$ and positive sign for $t \ge 1$. Hence by comparing derivatives, there exists positive constants $C_1, C_2$ such that for $\frac{1}{2}< t < \frac{3}{2}$ we have :
  \begin{equation} \label{eq:zetaApprox}
    C_1 < \abs*{\frac{\zeta(t)}{1-t}} < C_2.
  \end{equation}
In particular, for $\frac{1}{2}< t < \frac{3}{2}$ we have $-C_1 \le \zeta'(t) \le -C_2$ for suitable positive constants $C_1, C_2 > 0$. We also have the bound on the Airy function $\abs{\T{Ai}(x)} \le \min(1, 2/\sqrt{\pi}\abs{x}^\frac{1}{4})$ (\cite{NIST}, 9.7.9, 9.7.15). So for $n > N$ the asymptotic \eqref{eq:UnifSoftEdge} implies for $x \ge 2$,
  \begin{equation} \label{eq:SoftEdgeAiryBounds}
    \abs{\sqrt{n}\psi_{\nn}^{(\alpha)}(nx)} \le\min(
      2 (4\nn+2\alpha+2)^{\frac{1}{6}},\ \
      2 \pi^{-\frac{1}{2}} \abs{4-(x-2)^2}^{-\frac{1}{4}})
  \end{equation}
We now move to bounding the term $\psi_{n}^{(\alpha)}(nx) + \psi_{n-1}^{(\alpha)}(nx)$. In particular, the asymptotic of this sum, using \eqref{eq:UnifSoftEdge}, contains a difference of Airy functions. We will bound this difference by the derivative of the Airy function. In particular, we will need to give estimates for the difference of the arguments of the Airy functions. Let us now define,
  \begin{equation} \label{eq:t1t2Defs}
   t_1(x) = \frac{nx}{4n+2\alpha + 2},\ \ \ \ t_2(x) = \frac{nx}{4n+2\alpha-2}.
  \end{equation}
Because $t_2(x) = t_1(x)(1+O(n^{-1}))$,  $(4n+2\alpha-2)^{\frac{2}{3}} = (4n+2\alpha+2)^\frac{2}{3}(1-O(n^{-1}))$, and the constant bounds for $\zeta'(x)$, we have that for $\frac{1}{2} < t_1 < t_2 < \frac{3}{2}$:
  \[\zeta(t_2) = \zeta(t_1) - O(n^{-1}).\]
Note that $\zeta(t)$ is bounded for $\frac{1}{2} < t < \frac{3}{2}$. So for $\frac{1}{2} < t_1 < t_2 < \frac{3}{2}$,
  \begin{equation*}
  \begin{split}
   & \ \ \ \ \T{Ai}(-(4n+2\alpha+2)^\frac{2}{3} \zeta(t_1)) - \T{Ai}(-(4n+2\alpha-2)^\frac{2}{3} \zeta(t_2)) \\
   & = \T{Ai}(-(4n+2\alpha+2)^\frac{2}{3} \zeta(t_1)) - \T{Ai}(-(4n+2\alpha+2)^\frac{2}{3}\zeta(t_1) + O(n^{-\frac{1}{3}}))).
  \end{split}
  \end{equation*}
We also have the following bounds for $\T{Ai}'(x)$: $\abs{\T{Ai}'(x)} \le 2\frac{(-x)^{\frac{1}{4}}}{\sqrt{\pi}}$ when $x \le -1$ and $\abs{\T{Ai}'(x)} \le\frac{2}{\sqrt{\pi}}$ when $x \ge -1$ (\cite{NIST}, 9.7.10). 
Hence we can find sufficiently large constant $C$ so that for all $n > N$ and $\frac{1}{2} < t_1 < t_2 < \frac{3}{2}$:
  \begin{equation*}
  \begin{split}
     & \ \ \ \ \abs*{\frac{\Big[\T{Ai}(-(4n+2\alpha+2)^\frac{2}{3} \zeta(t_1)) - \T{Ai}(-(4n+2\alpha-2)^\frac{2}{3}\zeta(t_2))\Big]}{n^{-\frac{1}{3}}}} \\
     & \le \sup_{x\in [-(4n+2\alpha+2)^\frac{2}{3} \zeta(t_1), -(4n+2\alpha-2)^\frac{2}{3} \zeta(t_2)]} \frac{C}{2}\abs{\T{Ai}'(x)} \\
     & \le \begin{cases}
                  C\pi^{-\frac{1}{2}} [(4n+2\alpha+2)^\frac{2}{3} \zeta(t_1)]^{\frac{1}{4}}, & (4n+2\alpha+2)^\frac{2}{3} \zeta(t_1) \ge 1\\
                  C\pi^{-\frac{1}{2}}, & (4n+2\alpha+2)^\frac{2}{3} \zeta(t_1) \le 1.
               \end{cases}
  \end{split}
  \end{equation*}
Hence for $3<x<4+\epsilon$ and $n>N$ so that $\frac{1}{2} < t_1(x) < t_2(x) <\frac{3}{2}$, the asymptotic \eqref{eq:UnifSoftEdge} gives:
  \begin{equation} \label{eq:PsiSumBounds}
  \begin{split}
    & \ \ \ \ \abs{\sqrt{n}[\psi_{n}^{(\alpha)}(nx) + \psi_{n-1}^{(\alpha)}(nx)]} \\
    & \le \begin{cases}
        (1+O(n^{-1}))\frac{C\abs{1-t_1}^{\frac{1}{4}}}{\sqrt{\pi}x^{\frac{1}{4}}} + \frac{(4n+2\alpha+2)^{-\frac{1}{6}}}{x^{\frac{1}{4}}}O(n^{-\frac{2}{3}}), & (4n+2\alpha+2)^\frac{2}{3} \zeta(t_1) \ge 1 \\
        (1+O(n^{-1}))\frac{C(4n+2\alpha+2)^{-\frac{1}{6}}}{x^{\frac{1}{4}}}(1 + O(n^{-\frac{2}{3}})), & (4n+2\alpha+2)^\frac{2}{3} \zeta(t_1) \le 1.
  \end{cases}
  \end{split}
  \end{equation}
where $C$ is some sufficiently large positive constant. We note the $O(n^{-1})$ term in \eqref{eq:PsiSumBounds} comes directly from the asymptotics and have no dependence on anything but $\alpha$ and $n$. Combining \eqref{eq:zetaApprox}, \eqref{eq:SoftEdgeAiryBounds}, and \eqref{eq:PsiSumBounds} we get the bound for $n > N$ and $3 <x, y <4+\epsilon$:
  \begin{equation*} 
    n^2\psi_{\nn}^{(\alpha)}(ny)^2[\psi_{n}^{(\alpha)}(nx) + \psi_{n-1}^{(\alpha)}(nx)]^2 \le \begin{cases}
        C_1\frac{\abs{4-x}^{\frac{1}{2}}}{\abs{4-y}^\frac{1}{2}} + O(n^{-\frac{1}{2}}), & (4n+2\alpha+2)^\frac{2}{3} \zeta(t_1) \ge 1 \\
        C_2, & (4n+2\alpha+2)^\frac{2}{3} \zeta(t_1) \le 1.
        \end{cases}
  \end{equation*}
where $C_1$ and $C_2$ are positive constants not depending on $n$, $x$, or $y$ and the error $O(n^{-\frac{1}{2}})$ is uniform bounded in $x, y$. Taking the sum of these two bounds concludes the argument.
\end{proof}

\begin{lemma} \label{lem:hardEdgeBound}
 There exists $N$ such that for all $n > N$ and $x, y < 1$,
  \begin{equation*}
    n^2\psi_{\nn}^{(\alpha)}(ny)^2[\psi_{n}^{(\alpha)}(nx) - \psi_{n-1}^{(\alpha)}(nx)]^2 \le C_1 \Big(\frac{\sqrt{x}}{\sqrt{y}} + \frac{\sqrt{y}}{\sqrt{x}}\Big) + C_2,
  \end{equation*}
 where $C_1$ and $C_2$ are positive constants not depending on $n$, $x$, or $y$.
\end{lemma}
\begin{proof}
Note that the asymptotic of the difference $\psi_{n}^{(\alpha)}(nx) - \psi_{n-1}^{(\alpha)}(nx)$, using \eqref{eq:UnifHardEdge}, contains a difference of Bessel functions. We want to bound this difference by a derivative of Bessel functions. To do so, we need to provide estimates for how close the arguments of the two Bessel functions are. Recall the definition $\eta(t) = \frac{1}{2}(t-t^2)^{\frac{1}{2}} + \frac{1}{2}\arcsin \sqrt{t}$ from \eqref{eq:UnifHardEdge}. Note that $\frac{d}{dt}\eta(t) = \frac{\sqrt{1-t}}{\sqrt{t}}$. By comparing derivatives, we can see that there exists positive constants $C_1, C_2$ such that for $0 < t < \frac{3}{5}$,
  \begin{equation} \label{eq:etaBounds}
    C_1 < \frac{\eta(t)}{\sqrt{t}} < C_2
  \end{equation}
Let $t_1(x)$ and $t_2(x)$ be defined as in \eqref{eq:t1t2Defs}. We want to find bounds for the difference between $\eta(t_1)$ and $\eta(t_2)$. Let
  \begin{equation*}
    \omega = \frac{4nx}{(4n+2\alpha+2)(4n+2\alpha-2)},
  \end{equation*}
and write $t_2(x) = t_1(x) + \omega$. Then for $0 < t_1 \le t_2<\frac{2}{5}$:
\begin{equation*}
    (t_2 - t_2^2)^\frac{1}{2} = (t_1 - t_1^2)^\frac{1}{2}\Big(1 + \frac{2-4t_1}{(4n+2\alpha-2)(1-t_1)} + O(n^{-2}) \Big).
  \end{equation*}
We also have:
  \begin{equation*}
  \begin{split}
    \arcsin \sqrt{t_2} & = \arcsin \sqrt{t_1+\omega} = \arcsin \Big(\sqrt{t_1}\Big(1+\frac{\omega}{2} + O(n^{-2}) \Big)\Big).
  \end{split}
  \end{equation*}
Note that for $0 < t_1 < \frac{2}{5}$, we have $\sqrt{t_1} \le \arcsin \sqrt{t_1} \le 2\sqrt{t_1}$. In particular, for $0 < t_1 \le t_2 < \frac{2}{5}$ and $n > N$,
  \begin{equation*}
    \arcsin \sqrt{t_1} \le \arcsin \sqrt{t_2} \le \arcsin \sqrt{t_1} + \sqrt{t_1}\omega \le \arcsin \sqrt{t_1}(1+\omega).
  \end{equation*}
So we can conclude that $ \arcsin \sqrt{t_2} = \arcsin \sqrt{t_1} (1+O(n^{-1}))$ where the error $O(n^{-1})$ is bounded above by $\omega$ and bounded below by $0$. Note that $\eta'(t) > 0$ for $0 < t < \frac{2}{5}$. Thus for $0 < t_1 \le t_2 < \frac{2}{5}$ and $n > N$,
  \begin{equation} \label{eq:etaEstimate}
    \eta(t_2) = \eta(t_1)(1+O(n^{-1}))
  \end{equation}
where the error is bounded above by $2(\frac{2-4t_1}{(4n+2\alpha-2)(1-t_1)} + \omega)$ and below by $0$. The main point is that the bound on the $O(n^{-1})$ difference in \eqref{eq:etaEstimate} holds uniformly for $0 < t_1 \le t_2 < \frac{2}{5}$. In particular, for $x < 1$ we can find $N$ so that for all $n > N$,
  \[2\Big(\frac{2-4t_1}{(4n+2\alpha-2)(1-t_1)} + \omega\Big) < Cn^{-1},\]
where $C$ is some positive constant independent of $n$ and $x$. Hence for $x < 1$ and $n > N$,
  \begin{equation}\label{eq:BesselDerivative}
  \begin{split}
   & \ \ \ \ \eta(t_1)^{-1}[J_\alpha((4n+2\alpha+2) \eta(t_1)) - J_\alpha((4n+2\alpha-2) \eta(t_2))] \\
   & \le C \sup_{x\in [(4n+2\alpha+2) \eta(t_1), (4n+2\alpha-2) \eta(t_2)]} \abs{J_\alpha'(x)}.
  \end{split}
  \end{equation}
Using the relations $J_\alpha'(x) = \frac{J_{\alpha-1}(x) - J_{\alpha+1}(x)}{2}$ and $J_0'(x) = J_1(x)$ (\cite{NIST}, 10.6.1), we are able to simplify \eqref{eq:BesselDerivative}. The Bessel functions $\abs{J_\alpha(x)}$ ($\alpha \ge 0$, $x \ge 0$) are bounded above by $C_\alpha\sqrt{\frac{2}{\pi x}}$ where $C_\alpha$ is some constant depending on $\alpha$. Such a $C_\alpha$ exists because of the bound given by \cite{Landau2000} (which gives the uniform bound $\abs{J_\alpha(x)} \le \frac{c}{x^{1/3}}$ with $c \approx 0.79$) in conjunction with known estimates (see \cite{Watson1995}, pp. 206 or \cite{NIST}, 10.17.3). So we have the bound $\abs{J_{\alpha}'(x)} \le C_{\alpha}' x^{-\frac{1}{2}}$. Applying it to \eqref{eq:BesselDerivative} with the asymptotic \eqref{eq:UnifHardEdge}, we get for $x < 1$ and $n > N$,
  \begin{equation} \label{eq:HardEdgeEstimateDiff}
  \begin{split}
    \abs{\sqrt{n}[\psi_{n}^{(\alpha)}(nx) - \psi_{n-1}^{(\alpha)}(nx)]} \le C C_\alpha'\frac{x^{\frac{1}{4}}}{2} (1-t_1)^{-\frac{1}{4}} + \frac{x^\frac{1}{2}}{2} (1-t_1)^{-\frac{1}{4}}O(n^{-\frac{1}{2}}) + O(n^{-\frac{1}{2}}).
  \end{split}
  \end{equation}
Using the same Bessel bounds along with the asymptotics \eqref{eq:UnifHardEdge} we also get:
  \begin{equation} \label{eq:HardEdgeBoundsPsi}
    \abs{\sqrt{n}\psi_n^{(\alpha)}(nx)} \le \min(
      (4n + 2\alpha + 2)^{\frac{1}{2}}, C_\alpha\pi^{-\frac{1}{2}}(4x-x^2)^{-\frac{1}{4}})
  \end{equation}
Consequently, using \eqref{eq:etaBounds},\eqref{eq:HardEdgeEstimateDiff}, and \eqref{eq:HardEdgeBoundsPsi}, we get that for $x, y < 1$ and $n > N$,
  \begin{equation*}
  \begin{split}
    &\ \ \ \ n^2\psi_{\nn}^{(\alpha)}(ny)^2[\psi_{n}^{(\alpha)}(nx) - \psi_{n-1}^{(\alpha)}(nx)]^2 \\
    & \le (1 + O(n^{-1}))\frac{Cx^{\frac{1}{2}}}{y^\frac{1}{2}} (1-t_1)^{-\frac{1}{2}} + \frac{x^\frac{3}{4}}{y^\frac{1}{2}} (1-t_1)^{-\frac{1}{2}}O(n^{-\frac{1}{2}}) + \frac{Cx^{\frac{1}{4}}}{y^{\frac{1}{4}}} + O(1)\\
    & \le C_1 \Big(\frac{\sqrt{x}}{\sqrt{y}} + \frac{\sqrt{y}}{\sqrt{x}}\Big) + C_2.
  \end{split}
  \end{equation*}
where $C_1, C_2 > 0$ are positive constants not dependent on $n$, $x$, or $y$.
\end{proof}

With these two bounds we can now show the following convergence result. Note that the convergence is only exact on part of the domain $\Omega_1^{\epsilon_1}$ (more precisely only on $\Omega_1^0$). However it does provide us with a uniform bound for $\T{Var}^\epsilon(\mathcal{N}_n^\circ[f])$ in $n > N$, a fact which we will take full advantage of later. 
\begin{prop} \label{prop:BulkVarianceConvergence}
   For $f(x)\in L^\infty([0, 4])$ such that \eqref{eq:Assumption1} is satisfied for some $\epsilon > 0$, there exists $N$ such that the following is uniformly bounded in $n > N$:
  \begin{equation*}
    \T{Var}^\epsilon(\mathcal{N}_n^\circ[f]) := \int_0^{4+\epsilon} \int_0^{4+\epsilon} F(x, y)\Phi_n^{(\alpha)}(x, y) dx dy,
  \end{equation*}
  In particular this implies that
  \begin{equation*}
    \lim_{n\to\infty} \int_0^4 \int_0^4 F(x, y)\Phi_n^{(\alpha)}(x, y) dx dy = \int_0^4\int_0^4 F(x, y)\Xi(x, y) dx dy,
  \end{equation*}
  where $F(x, y)$, $\Phi_n^{(\alpha)}(x,y)$, and $\Xi(x, y)$ are defined in \eqref{eq:FDefs}, \eqref{eq:PhiDefs}, and \eqref{eq:XiDefs}. 
\end{prop}
\begin{proof}
It will be convenient to go back to the representation of the variance given by \eqref{eq:varKernalRep}. In particular, we will expand the numerator of $K_n^{(\alpha)}(x,y)$ as:
  \begin{equation*}
  \begin{split}
      (x-y)K_n^{(\alpha)}(x,y)& = \psi_{n-1}^{(\alpha)}(nx)[\psi_{n}^{(\alpha)}(ny) + \psi_{n-1}^{(\alpha)}(ny)] - \psi_{n-1}^{(\alpha)}(ny)[\psi_{n}^{(\alpha)}(nx)+\psi_{n-1}^{(\alpha)}(nx)] \\
      & = \psi_{n-1}^{(\alpha)}(nx)[\psi_{n}^{(\alpha)}(ny) - \psi_{n-1}^{(\alpha)}(ny)] - \psi_{n-1}^{(\alpha)}(ny)[\psi_{n}^{(\alpha)}(nx)-\psi_{n-1}^{(\alpha)}(nx)].
  \end{split}
  \end{equation*}
where the first line will be needed for the soft-edge while the second for the hard edge. Let $\Sigma_1 = [0,1)^2 \cup [3, 4+\epsilon)^2$. Then Lemma \ref{lem:softEdgeBound} and Lemma \ref{lem:hardEdgeBound} tells us there exists $N$ so that for all $n > N$ and $(x, y) \in \Sigma_1$ and we have,
\begin{equation*}
  \begin{split}
    (x-y)^2K_n^{(\alpha)}(x,y)^2 \le C\bigg[\Big(\frac{\sqrt{\abs{x}}}{\sqrt{\abs{y}}} + \frac{\sqrt{\abs{y}}}{\sqrt{\abs{x}}}\Big) + \Big(\frac{\sqrt{\abs{4-x}}}{\sqrt{\abs{4-y}}} + \frac{\sqrt{\abs{4-y}}}{\sqrt{\abs{4-x}}}\Big) + 1\bigg].
  \end{split}
  \end{equation*} 
where $C>0$ is some sufficiently large constant. Let us denote:
  \begin{equation}\label{eq:BarXiDefs}
    \tilde{\Xi}(x, y) := \frac{1}{8\pi^2}\Big(\frac{\sqrt{\abs{(4-y)x}}}{\sqrt{\abs{(4-x)y}}} + \frac{\sqrt{\abs{(4-x)y}}}{\sqrt{\abs{(4-y)x}}}\Big)
  \end{equation}
Then we get the inequalities,
  \begin{equation*}
   \tilde{\Xi}(x, y) \ge \begin{cases} \frac{1}{8\pi^2}\Big(\frac{\sqrt{x}}{\sqrt{y}} + \frac{\sqrt{y}}{\sqrt{x}}\Big), & 0 \le x, y < 2 \\
   \frac{1}{8\pi^2}\Big(\frac{\sqrt{\abs{4-x}}}{\sqrt{\abs{4-y}}} + \frac{\sqrt{\abs{4-y}}}{\sqrt{\abs{4-x}}}\Big), & 2 < x, y \\
   \frac{1}{4\pi^2}, & 0 \le x, y \le 4 + \epsilon
                       \end{cases}
  \end{equation*}
Hence there exists a positive constant $C$ so that for $n > N$,
  \begin{equation*}
  \begin{split}
    \iint_{\Sigma_1} F(x, y)\Phi_n^{(\alpha)}(x, y) dx dy & \le C\iint_{\Sigma_1} F(x, y)\tilde{\Xi}(x, y) dx dy \le C V_{\T{LUE}}^\epsilon[f].
  \end{split}
  \end{equation*}
For $(x, y)\in \Sigma_2 = [\delta ,4-\delta]^2$ we have the bound \eqref{eq:bulkBoundPhi}. Hence for $n > N$,
  \begin{equation*}
  \begin{split}
    \iint_{\Sigma_2} F(x, y)\Phi_n^{(\alpha)}(x, y) dx dy & \le C\iint_{\Sigma_2}F(x, y)\tilde{\Xi}(x, y)dx dy \le C V_{\T{LUE}}^\epsilon[f],
  \end{split}
  \end{equation*}
where $C>0$ is a positive constant. Finally for $(x, y) \in \Sigma_3 = [0,4+\epsilon)^2\setminus (\Sigma_1\cup \Sigma_2)$ we have that $F(x, y)$ is $L^\infty$. Hence Proposition \ref{prop:Double} tells us that for $n > N$, 
  \begin{equation*}
  \begin{split}
    \iint_{\Sigma_3} F(x, y)\Phi_n^{(\alpha)}(x, y) dx dy & \le 2 \iint_{\Sigma_3} F(x, y)\tilde{\Xi}(x, y) dx dy \le 2V_{\T{LUE}}^\epsilon[f].
  \end{split}
  \end{equation*}
Hence we have the uniform bound for $n > N$,
  \begin{equation*}
    \T{Var}^\epsilon(\mathcal{N}_n^\circ[f]) = \int_0^{4+\epsilon} \int_0^{4+\epsilon} F(x, y)\Phi_n^{(\alpha)}(x, y) dx dy \le C V_{\T{LUE}}^\epsilon[f].
  \end{equation*}
For the second part, because of the continuity of integrals, given $\nu > 0$ we can find $\delta$ such that for all $n > N$,
  \begin{equation} \label{eq:BulkBoundaryEstimate}
    \iint_{[0,4]^2 \setminus [\delta, 4-\delta]^2}F(x, y)\Phi_n^{(\alpha)}(x, y) dx dy \le C\iint_{[0,4]^2 \setminus [\delta, 4-\delta]^2}F(x, y)\Xi(x, y) < \frac{\nu}{2},
  \end{equation}
where we recall $\Xi$ is given by \eqref{eq:XiDefs}. Then Lemma \ref{lem:BulkVariance} tells us that for $n > N$,
  \begin{equation} \label{eq:BulkInsideEstimate}
   \bigg\lvert \iint_{[\delta, 4-\delta]^2} F(x, y)\Phi_n^{(\alpha)}(x, y) dx dy - \int_0^4\int_0^4 F(x, y)\Xi(x, y)dx dy \bigg\rvert < \frac{\nu}{2}.
  \end{equation}
Adding \eqref{eq:BulkBoundaryEstimate} and \eqref{eq:BulkInsideEstimate}, we see that for $n > N$,
  \begin{equation*}
  \bigg\lvert \int_0^4 \int_0^4 F(x, y)\Phi_n^{(\alpha)}(x, y) dx dy - \int_0^4\int_0^4 F(x, y)\Xi(x, y) dx dy \bigg\rvert < \nu.
  \end{equation*}
\end{proof}

\subsection{Convergence in \texorpdfstring{$\Omega_2^{\epsilon_2}$}{Omega2}} 
In $\Omega_2^{\epsilon_2}$, the problem arises once again in a neighborhood of $x=y$, as when $x$ and $y$ are apart, $F(x, y)$ is $L^\infty$ and Proposition \ref{prop:Double} will suffice. The following result shows that the contribution from the region where $x$ is close to $y$ does not blow up, and tends to $0$. 
\begin{prop} \label{prop:outsideConvergenceVariance}
For $f\in L^\infty([4+\epsilon, \infty))$ where $\epsilon > 0$ we have,
  \begin{equation*}
    \lim_{n\to\infty }\int_{4+\epsilon}^\infty \int_{4+\epsilon}^\infty F(x, y) \Phi_n^{(\alpha)}(x, y) dx dy = 0,
  \end{equation*}
 where $F(x, y)$ and $\Phi_n^{(\alpha)}(x, y)$ are defined in \eqref{eq:FDefs} and \eqref{eq:PhiDefs}.
\end{prop}
\begin{proof}
Define $\Sigma_{\delta} = \{(x,y)\ |\ x \in B(y, \delta), y\in [4+\epsilon, \infty)\}$ and $\Sigma_\delta^c = [4+\epsilon,\infty)^2\setminus \Sigma_{\delta}$. We want to show that $K_n^{(\alpha)}(x,y)^2$ is in $L^1([4+\epsilon, \infty)^2)$, then H\"{o}lder's inequality with dominated convergence theorem would conclude the proof. 

We start with the soft-edge asymptotic \eqref{eq:UnifSoftEdge}. In particular, for $t \ge 1$, the term inside the Airy function is $-(\nn+2\alpha+2)^{\frac{2}{3}}\zeta(t) \ge 0$.  We also have the bound on the Airy function $\T{Ai}(\abs{x}) \le \exp(-\frac{2}{3}\abs{x}^{\frac{3}{2}})/(2\sqrt{\pi} \abs{x}^{\frac{1}{4}})$, (\cite{NIST} 9.7.15). Now note that given $\epsilon > 0$, for $n > N$ we have that $x \ge 4+\epsilon$ implies $t(x) \ge 1+\frac{\epsilon}{8}$. This gives the bound $\abs{1-t}^{-\frac{1}{4}} \le 2\abs{4-x}^{-\frac{1}{4}}$ for $x \ge 4+\epsilon$ and $n > N$. In particular, this implies for $x \ge 4+\epsilon$ and $n > N$, with the previous Airy bound:
  \begin{equation*}
    \abs*{\sqrt{n}\psi_{\nn}^{(\alpha)}(nx)} \le \frac{2\exp(-\frac{2}{3}(4\nn+2\alpha+2) \abs{\zeta(t)}^{\frac{3}{2}})}{\abs{4-(x-2)^2}^{\frac{1}{4}}}.
  \end{equation*}
From \eqref{eq:zetaDefs}, for $t \ge 1$ we have that $\gamma(t) := \frac{2}{3}\abs{\zeta(t)}^{\frac{3}{2}} = \frac{1}{2}[(t^2-t)^{\frac{1}{2}} - \arccosh(\sqrt{t})] \ge 0$. Note that the derivative $\gamma'(t) = \sqrt{\frac{t-1}{t}}$ is strictly increasing for $t > 1$. 
So we can further simplify our bound, for $x \ge 4+\epsilon$ and $n > N$,
  \begin{equation} \label{eq:psiExpBound}
    \abs*{\sqrt{n}\psi_{\nn}^{(\alpha)}(nx)}  \le \frac{\exp(-(4\nn+2\alpha+2) \gamma(t))}{\abs{x(1-t)}^{\frac{1}{4}}} \le \frac{\exp(-\gamma(t))}{\abs{4-(x-2)^2}^{\frac{1}{4}}}.
  \end{equation}
Now, for $x \ge 4+\epsilon$ and $n > N$ we can bound $\gamma(t)$ from below by a line with slope $\gamma'(1+\frac{\epsilon}{8})$: 
  \begin{equation} \label{eq:gammaBound}
    \gamma(t) \ge \sqrt{\frac{\epsilon/8}{1+\epsilon/8}} (x-4) + \Big(\gamma(1+\epsilon/8)- \sqrt{\frac{\epsilon/8}{1+\epsilon/8}} \epsilon\Big)\ge \gamma(1+\epsilon/8) > 0.
  \end{equation}
Let $C_\epsilon = -(\gamma(1+\epsilon/8)- \sqrt{\frac{\epsilon/8}{1+\epsilon/8}} \epsilon) >0$. Fixing $y$, $\abs{x - y} < \delta$ and $\delta$ much less than $\epsilon$. We can then use the recurrence \eqref{eq:psiDerivativeRecurrence} so that for $n > N$, and $x, y > 4+\epsilon$,
  \begin{equation} \label{eq:TempEq1}
  \begin{split}
    K_n^{(\alpha)}(x,y)^2 & = n(n+\alpha)\Big(\frac{\psi_{n-1}^{(\alpha)}(nx)\psi_{n}^{(\alpha)}(ny) - \psi_{n}^{(\alpha)}(nx)\psi_{n-1}^{(\alpha)}(ny)}{x-y}\Big)^2  \\
    & \le 2n^2\Big(\abs*{\psi_n^{(\alpha)}(ny)\sup_{\abs{x-y} < \delta} n\psi_{n-1}^{(\alpha)\prime}(nx)} + \abs*{\psi_{n-1}^{(\alpha)}(ny)\sup_{\abs{x-y} < \delta} n\psi_{n}^{(\alpha)\prime}(nx)}\Big)^2 \\
      & \le 2n^2\bigg(4\abs*{\psi_n^{(\alpha)}(ny)\sup_{\abs{x-y} < \delta} \psi_{n-1}^{(\alpha)}(nx)} + 2\abs*{\psi_n^{(\alpha)}(ny)\sup_{\abs{x-y} < \delta} \psi_{n-2}^{(\alpha)}(nx)}  \\
      & \qquad + 4\abs*{\psi_{n-1}^{(\alpha)}(ny)\sup_{\abs{x-y} < \delta} \psi_{n}^{(\alpha)}(nx)} + 2\abs*{\psi_{n-1}^{(\alpha)}(ny)\sup_{\abs{x-y} < \delta} \psi_{n-1}^{(\alpha)}(nx)}\bigg)^2.
  \end{split}
  \end{equation}
Applying \eqref{eq:psiExpBound}, for $x, y \ge 4+\epsilon$ and $n > N$, \eqref{eq:TempEq1} is bounded above by,
  \begin{equation*}
   2\Big(C \frac{\exp(-\gamma(t(y-\delta))}{{\abs{4-(y-\delta-2)^2}^{\frac{1}{4}}}}\Big)^4 \le 2C^4 \frac{\exp(-\gamma(t(y-\delta))}{{\abs{4-(y-\delta-2)^2}}} \le \frac{1}{\epsilon} C^4\exp(-\gamma(t(y-\delta))),
  \end{equation*}
where $C>0$ is some positive constant. Now we use the linear lower bound for $\gamma(t)$ from \eqref{eq:gammaBound} along with the fact that $t(y)$ is bounded below by $\frac{y}{4} - \delta$ for $n > N$:
  \begin{equation*}
  \begin{split}
   \frac{1}{\epsilon} C\exp(-\gamma(t(y-\delta))) & \le \frac{1}{\epsilon} C\exp\Big(-\gamma'\Big(1+\frac{\epsilon}{8}\Big)\Big(\frac{y-\delta}{4} - \delta\Big)\Big) \exp(C_\epsilon) \\
    & \le \frac{1}{\epsilon} C\exp\Big(-\gamma'\Big(1+\frac{\epsilon}{8}\Big)\frac{y}{4}\Big) \exp\Big(\gamma'\Big(1+\frac{\epsilon}{8}\Big)\frac{5\delta}{4}+C_\epsilon\Big).
  \end{split}
  \end{equation*}
So for $n > N$ and $\delta$ much less than $\epsilon$, we get:
  \begin{equation*}
  \begin{split}
   & \ \ \ \ \iint_{\Sigma_\delta} F(x, y) \Phi_n^{(\alpha)}(x, y) dx dy \\
   & \le \frac{C}{\epsilon} \iint_{\Sigma_\delta} (f(x)-f(y))^2 \exp\Big(-\gamma'\Big(1+\frac{\epsilon}{8}\Big)\frac{y}{4}\Big)  \exp\Big(\gamma'\Big(1+\frac{\epsilon}{8}\Big)\frac{5\delta}{4}+C_\epsilon\Big) dx dy \\
   & \le \frac{4\norm{f}_\infty^2}{\epsilon} \iint_{\Sigma_\delta} C\exp\Big(-\gamma'\Big(1+\frac{\epsilon}{8}\Big)\frac{y}{4}\Big) \exp\Big(\gamma'\Big(1+\frac{\epsilon}{8}\Big)+C_\epsilon\Big) dx dy \\
   & < \infty.
  \end{split}
  \end{equation*}
So by the dominated convergence theorem, we conclude that uniformly in $n$ for $n > N$,
  \begin{equation*}
   \lim_{\delta \to 0} \iint_{\Sigma_\delta} F(x, y) \Phi_n^{(\alpha)}(x, y) dx dy = 0.
  \end{equation*}
Thus given $\nu > 0$ we can find $\delta > 0$ such that 
  \begin{equation} \label{eq:deltaBound}
   \abs*{\iint_{\Sigma_\delta} F(x, y) \Phi_n^{(\alpha)}(x, y) dx dy} < \frac{\nu}{2}.
  \end{equation}
In $\Sigma_\delta^c$ we have that $\abs{x-y} \ge \delta$ and hence $F(x, y)$ is $L^\infty$. Therefore by Proposition \ref{prop:Double}, we can find $n > N$ such that:
  \begin{equation} \label{eq:deltaCBound}
    \abs*{\iint_{\Sigma_\delta^c} F(x, y) \Phi_n^{(\alpha)}(x, y) dx dy} < \frac{\nu}{2}.
  \end{equation}
Adding \eqref{eq:deltaBound} and \eqref{eq:deltaCBound} we conclude that for $n > N$,
  \begin{equation*}
    \abs*{\int_{4+\epsilon}^\infty \int_{4+\epsilon}^\infty F(x, y) \Phi_n^{(\alpha)}(x, y) dx dy} < \nu.
  \end{equation*}
\end{proof}

\subsection{Proof of Theorem \ref{thm:VarConvergence}}
The proof simply combines Propositions \ref{prop:Double}, \ref{prop:BulkVarianceConvergence}, and \ref{prop:outsideConvergenceVariance} using some slight manipulations of $\epsilon$.
\begin{proof}
Let $\nu > 0$ be given. Let $f\in L^\infty([0,\infty))$ such that there exists $\epsilon > 0$ satisfying assumption \eqref{eq:Assumption1}. Then from Proposition \ref{prop:BulkVarianceConvergence} there exists $N$ such that for $n > N$,
  \begin{equation*}
    \T{Var}^\epsilon(\mathcal{N}_n^\circ[f]) := \int_0^{4+\epsilon} \int_0^{4+\epsilon} F(x, y)\Phi_n^{(\alpha)}(x, y) dx dy.
  \end{equation*}
is uniformly bounded in $n > N$ by some constant $C$. As $\Phi_n^{(\alpha)}$ is bounded uniformly for $n > N$ by $\bar{\Xi}(x, y)$ \eqref{eq:BarXiDefs}, condition \eqref{eq:Assumption1} tells us we can find $0 < \epsilon_1 < \epsilon$ so that:
  \begin{equation} \label{eq:OuterEdges}
    \T{Var}^{\epsilon_1}(\mathcal{N}_n^\circ[f]) - \T{Var}^{0}(\mathcal{N}_n^\circ[f]) = \iint_{[0, 4+\epsilon_1)^2 \setminus [0,4]^2} F(x, y)\Phi_n^{(\alpha)}(x, y) dx dy < \frac{\nu}{4}.
  \end{equation}
By Proposition \ref{prop:BulkVarianceConvergence} for $n > N$ we have,
  \begin{equation} \label{eq:Bulks}
  \abs*{\int_0^4 \int_0^4 F(x, y)\Phi_n^{(\alpha)}(x, y) dx dy - \int_0^4\int_0^4 F(x, y)\Xi(x, y) dx dy} < \frac{\nu}{4}.
  \end{equation}
Finally from Proposition \ref{prop:outsideConvergenceVariance} we can find $N$ so that $n > N$: 
  \begin{equation} \label{eq:Outsides}
    \int_{4+\frac{\epsilon_1}{2}}^\infty \int_{4+\frac{\epsilon_1}{2}}^\infty F(x, y)\Phi_n^{(\alpha)}(x, y) dx dy < \frac{\nu}{4}.
  \end{equation}
We choose $\Omega_3 = [0, \infty)^2\setminus ([0, 4+\epsilon_1)^2 \cup [4+\frac{\epsilon_1}{2}, \infty)^2)$. Essentially we are working with regions $\Omega_1^{\epsilon_1}$ and $\Omega_2^{\epsilon_1/2}$ as defined in the beginning of Section \ref{sec:Thm1Proof}. Note that $F(x, y)$ is $L^\infty$ in $\Omega_3$. So Proposition \ref{prop:Double} tells us for $n > N$:
  \begin{equation} \label{eq:CornerRectangles}
    \iint_{\Omega_3} F(x, y)\Phi_n^{(\alpha)}(x, y) dx dy < \frac{\nu}{4}.
  \end{equation}
Adding \eqref{eq:OuterEdges}, \eqref{eq:Bulks}, \eqref{eq:Outsides}, and \eqref{eq:CornerRectangles}, we see that for $n > N$:
  \begin{equation}
    \abs*{\int_0^\infty \int_0^\infty F(x, y)\Phi_n^{(\alpha)}(x, y) dx dy - \int_0^4\int_0^4 F(x, y)\Xi(x, y) dx dy} < \nu.
  \end{equation}
\end{proof}

\section{Approximation by Chebyshev polynomials (proof of Lemma \ref{lem:VLUEApprox})}  \label{sec:Lem1Proof}
The key idea is that we can show the Chebyshev-$\dot{H}^\frac{1}{2}$ norm \eqref{defs:ChebyshevNorm} of an $L^\infty$ function is finite iff $V_{\T{GUE}}[f]$ \eqref{eq:GUEVar} is finite. This is stated more concretely in Lemma \ref{lem:seminormsEquiv}. To show this claim we follow a proof similar to that of (\cite{Lieb2001}, Thm. 7.12) where instead of the Poisson kernel, for $\dot{H}^\frac{1}{2}(\RR)$ we use the one corresponding to Chebyshev polynomials of the first kind. We also note that up to change of variables, $V_{\T{LUE}} = V_{\T{GUE}}$. As a result, condition \eqref{eq:Assumption1} is sufficient for us to use a Chebyshev approximation. 
\subsection{Definitions and relevant background}
Recall the limiting variance of the centred linear spectral statistic for the GUE (\cite{Pastur2011}, Remark 2.1):
  \begin{equation}\label{eq:GUEVar}
      V_{\T{GUE}}[f] := \int_{-2}^2 \int_{-2}^2\Big(\frac{f(x)-f(y)}{x-y}\Big)^2 \frac{4-xy}{4\pi^2 \sqrt{4-x^2}\sqrt{4-y^2}} dx dy.
  \end{equation}
Now let us define the operator:
  \begin{equation} \label{defs:KOperator}
   Kf(x) := \int_{-2}^2 \frac{f(x)-f(y)}{(x-y)^2} \frac{4-xy}{2\pi\sqrt{4-y^2}} dy,
  \end{equation}
where the integral is evaluated in principal value. Let $L_w^2([-2,2])$ be the weighted function space defined by the inner product,
  \[\inn{f}{g}_w = \int_{-2}^2 \frac{f(x)g(x)}{2\pi\sqrt{4-x^2}} dx.\]
Let $T_n$ be the Chebyshev polynomials of the first kind, which satisfies the relations:
  \begin{equation} \label{eq:recChebyshev}
   T_{n+1}(x) = 2xT_n(x) - T_{n-1}(x),\ \ T_n(\cos\theta) = \cos n\theta.
  \end{equation}
It will be convenient to define $P_n(x) = T_n(x/2)$. In particular the first few terms are $P_0(x) = 1$, $P_1(x) = x/2$, and $P_2(x) = \frac{x^2}{2}-1$. They also satisfy the orthogonality relation,
  \begin{equation} \label{eq:orthoChebyshev}
   \int_{-2}^2 \frac{P_n(x)P_m(x)}{\pi \sqrt{4-x^2}} dx = \begin{cases}                                         0 & n \ne m \\                                    1 & n = m = 0 \\                                                              \frac{1}{2} & n = m \ne 0                                                           \end{cases}
  \end{equation}
We define the Chebyshev-$\dot{H}^\frac{1}{2}$ norm by:
  \begin{equation} \label{defs:ChebyshevNorm}
        \norm{f}_{\dot{H}^\frac{1}{2}_C} := \sum_{k=1}^\infty na_n^2, \T{ where } a_n = \int_{-2}^2 \frac{2f(x) P_n(x)}{\pi\sqrt{4-x^2}} dx.
  \end{equation}
We will also need the following identity (evaluated in principal value):
  \begin{equation} \label{eq:pvIdentity}
    \int_{-2}^2 \frac{1}{x-y} \frac{1}{\sqrt{4-y^2}} dy = 0.
  \end{equation}
We now show some properties of $K$ that will be useful. We follow a method similar to Lemma 3.1 in \cite{Yau2015}) to prove the following:
\begin{lemma} \label{lem:KKernal}
 $KP_n = \frac{n}{2}P_n$ and in particular the kernel of $e^{-tK}$ is given by
  \begin{equation*}
  \begin{split}
   K_t(x,y) & = 2 + 4\sum_{n=1}^\infty  e^{-\frac{tn}{2}} P_n(x)P_n(y) \\
    & = \frac{1-e^{-t}}{1+e^{-t}-2\cos(\theta+\phi)e^{-\frac{t}{2}}} + \frac{1-e^{-t}}{1+e^{-t}-2\cos(\theta-\phi)e^{-\frac{t}{2}}},
  \end{split}
  \end{equation*}
where $\frac{x}{2} = \cos\theta, \frac{y}{2} = \cos\phi$. . 
\end{lemma}
\begin{proof} We will use induction. The initial case $KP_0(x) = 0$ is clear. From \eqref{eq:pvIdentity} we also have:
  \begin{equation*}
    KP_1(x) = \frac{4-x^2}{4\pi} \int_{-2}^2 \frac{1}{x-y} \frac{1}{\sqrt{4-y^2}} dy + \frac{x}{4\pi} \int_{-2}^2 \frac{1}{\sqrt{4-y^2}} dy= \frac{1}{2}P_1(x).
  \end{equation*}
Now assuming that $KP_n = \frac{n}{2}P_n$ holds up to $n$, using the recurrence \eqref{eq:recChebyshev} we get:
  \begin{equation*}
  \begin{split}
    KP_{n+1}(x) & = \int_{-2}^2 \frac{xP_n(x) - P_{n-1}(x) - (yP_n(y) - P_{n-1}(y))}{(x-y)^2} \frac{4-xy}{2\pi\sqrt{4-y^2}}dy \\
      & = xKP_n(x) - KP_{n-1}(x) + \int_2^2 \frac{P_n(y)}{x-y} \frac{4-xy}{2\pi\sqrt{4-y^2}}dy \\
      & = \frac{n+1}{2}xP_n(x) - \frac{n-1}{2}P_{n-1}(x) - \int_{-2}^2 \frac{P_n(x) - P_n(y)}{x-y} \frac{4-xy}{2\pi\sqrt{4-y^2}}dy \\
      & = \frac{n+1}{2}xP_n(x) - \frac{n+1}{2}P_{n-1}(x) = \frac{n+1}{2}P_{n+1}(x).
  \end{split}
  \end{equation*}
The second-last equality is nontrivial. For $n \ge 1$, we need to show that,
  \begin{equation} \label{eq:RecPol1}
   \int_{-2}^2 \frac{P_n(x) - P_n(y)}{x-y} \frac{4-xy}{2\pi\sqrt{4-y^2}}dy = P_{n-1}(x).
  \end{equation}
We will once again use induction here. In particular we have two base cases which both follow from \eqref{eq:orthoChebyshev}:
  \begin{equation*}
    \int_{-2}^2 \frac{P_1(x)-P_1(y)}{x-y} \frac{4-xy}{2\pi\sqrt{4-y^2}}dy = \int_{-2}^2 \frac{4-xy}{4\pi\sqrt{4-y^2}}dy =P_0(x).
  \end{equation*}
And similarly,
  \begin{equation*}
    \int_{-2}^2 \frac{P_2(x)-P_2(y)}{x-y} \frac{4-xy}{2\pi\sqrt{4-y^2}}dy = \int_{-2}^2 (x+y)\frac{4-xy}{4\pi\sqrt{4-y^2}}dy =P_1(x).
  \end{equation*}
Suppose \eqref{eq:RecPol1} holds up to $n \ge 2$, then:
  \begin{equation*}
  \begin{split}
    & \ \ \ \ \int_{-2}^2 \frac{P_{n+1}(x) - P_{n+1}(y)}{x-y} \frac{4-xy}{2\pi\sqrt{4-y^2}}dy \\
    & = x\int_{-2}^2 \frac{P_n(x) - P_n(y)}{x-y} \frac{4-xy}{2\pi\sqrt{4-y^2}}dy + \int_{-2}^2 P_n(y) \frac{4-xy}{2\pi\sqrt{4-y^2}}dy  \\
      & \qquad - \int_{-2}^2 \frac{P_{n-1}(x) - P_{n-1}(y)}{x-y}\frac{4-xy}{2\pi\sqrt{4-y^2}}dy \\
      & = xP_{n-1}(x) - P_{n-2}(x) + \int_{-2}^2 P_n(y) \frac{4-xy}{2\pi\sqrt{4-y^2}}dy \\
      & = xP_{n-1}(x) - P_{n-2}(x) = P_{n}(x),
  \end{split}
  \end{equation*}
where for the second-last equality we used the orthogonal property of $P_n(x)$ \eqref{eq:orthoChebyshev}, as $n \ge 2$. Hence, we have shown that $KP_n(x) = \frac{n}{2}P_{n}(x)$. 

Now, note that $(e^{-tK}P_n)(x) = e^{-\frac{tn}{2}}P_n(x)$. Hence the kernel of $e^{-tK}$ is 
  \[K_t(x, y) = 2 + 4\sum_{n=1}^\infty e^{-\frac{tn}{2}}P_n(x)P_n(y).\]
Using the equivalent definition of Chebyshev polynomials in terms of $\cos$ \eqref{eq:recChebyshev} and substituting de Moivre's identity $(\cos x + i\sin x)^n = \cos nx + i \sin nx$, we get that:
  \begin{equation*}
  \begin{split}
   K_t(x,y) = \frac{1-e^{-t}}{1+e^{-t}-2\cos(\theta+\phi)e^{-\frac{t}{2}}} + \frac{1-e^{-t}}{1+e^{-t}-2\cos(\theta-\phi)e^{-\frac{t}{2}}}.
  \end{split}
  \end{equation*}
where $\frac{x}{2} = \cos\theta, \frac{y}{2} = \cos\phi$.
\end{proof}

We will also need the following fact about the kernel $K_t$:
\begin{lemma} \label{lem:KKernalNegative} The kernel defined in Lemma \ref{lem:KKernal} grows slower than $t$. That is to say,
  \begin{equation*}
    \frac{K_t(x, y)}{t}
  \end{equation*}
is strictly decreasing for $t > 0$.
\end{lemma}
\begin{proof}
 We take the derivative of $K_t(x, y)$ with respect to $t$:
  \begin{equation*}
    \frac{d}{dt} \frac{K_t(x,y)}{t}  = \frac{-1}{t^2}\Big(K_t(x,y) + \frac{t}{2} \Big(4\sum_{n=1}^\infty ne^{-\frac{tn}{2}}P_n(x)P_n(y)\Big)\Big).
  \end{equation*}
It would suffice to show that for $t > 0$:
  \begin{equation*}
   K_t(x,y) + \frac{t}{2} \Big(4\sum_{n=1}^\infty ne^{-\frac{tn}{2}}P_n(x)P_n(y)\Big) > 0.
  \end{equation*}
Note that $\sum_{n=1}^\infty nx^n = \frac{x}{(1-x)^2}$. So employing the same technique as before of expanding $P_n(x)$ using de Moivre's formula, and substituting $\frac{x}{2} = \cos\theta$ and $\frac{y}{2} = \cos\phi$:
  \begin{equation*}
    \frac{t}{2} \Big(4\sum_{n=1}^\infty ne^{-\frac{tn}{2}}P_n(x)P_n(y)\Big) = \sum_{n=0, 1}\frac{t}{2} \Big(\frac{2\cos(\theta + (-1)^n\phi)(e^{-\frac{t}{2}} + e^{-\frac{3t}{2}}) - 4e^{-t}}{(1+e^{-t} - 2\cos(\theta + (-1)^n\phi)e^{-\frac{t}{2}})^2}\Big).
  \end{equation*}
We compare $K_t(x, y)$ and the above sum term-by-term using the second representation of $K_t(x,y)$ from Lemma \ref{lem:KKernal}. In particular we want to show that
  \begin{equation} \label{eq:DerivativeKernal}
   \frac{1-e^{-t}}{1+e^{-t}-2\cos(\theta\pm\phi)e^{-\frac{t}{2}}} + \frac{t}{2}\frac{2\cos(\theta \pm \phi)(e^{-\frac{t}{2}} + e^{-\frac{3t}{2}}) - 4e^{-t}}{(1+e^{-t} - 2\cos(\theta \pm \phi)e^{-\frac{t}{2}})^2} > 0.
  \end{equation}
Combining the terms of \eqref{eq:DerivativeKernal} so that the denominator is a square, it suffices to show, for $t \ge 0$, that the numerator is positive. This is to say we want to show:
  \begin{equation*}
    1-e^{-2t} - 2t + (2e^{-\frac{3t}{2}} - 2e^{-\frac{t}{2}} + te^{-\frac{t}{2}} + te^{-\frac{3t}{2}})\cos(\theta \pm \phi) \ge 0,
  \end{equation*}
with equality only if $t = 0$. Note that:
  \begin{equation*}
    2e^{-\frac{3t}{2}} - 2e^{-\frac{t}{2}} + te^{-\frac{t}{2}} + te^{-\frac{3t}{2}} = e^{-\frac{t}{2}}(2e^{-t} - 2 + t + te^{-t}) \ge 0.
  \end{equation*}
So we see that,
  \begin{equation*}
   1-e^{-2t} - 2t + (2e^{-\frac{3t}{2}} - 2e^{-\frac{t}{2}} + te^{-\frac{t}{2}} + te^{-\frac{3t}{2}})\cos(\theta \pm \phi) \ge (1-te^{-\frac{t}{2}} - e^{-t})(1+e^{-\frac{t}{2}})^2.
  \end{equation*}
Since $1-te^{-\frac{t}{2}} - e^{-t} \ge 0$, with equality iff $t = 0$, we conclude:
  \begin{equation}
   K_t(x,y) + \frac{t}{2} \Big(4\sum_{n=1}^\infty ne^{-\frac{tn}{2}}P_n(x)P_n(y)\Big) > 0
  \end{equation}
for $t > 0$. The claim follows directly.
\end{proof}

\subsection{Equivalence of Chebyshev-\texorpdfstring{$\dot{H}^\frac{1}{2}$}{H1/2} and \texorpdfstring{$V_{\T{GUE}}$}{VGUE}}
We note that if $f\in L^2_w([-2, 2])$ then its Chebyshev series converges to $f$ in $L^2_w$. To see this use the change of variable $x = 2\cos\theta$ to get:
  \begin{equation*}
    \int_{-2}^2 \frac{f(x)^2}{\sqrt{4-x^2}}dx = \int_0^\pi f(2\cos\theta)^2 d\theta.
  \end{equation*}
Hence $f(x)\in L^2_w([-2, 2])$ is equivalent to saying $f(2\cos\theta)\in L^2([0, \pi])$. We can express the coefficients, for $n \ge 1$,
  \begin{equation*}
   a_n(f) = \int_{-2}^2 \frac{2f(x)P(x)}{\pi \sqrt{4-x^2}} dx = \frac{2}{\pi}\int_0^\pi f(2\cos\theta) \cos n\theta d\theta.
  \end{equation*}
In particular, $2a_0 + \sum_{n=1}^\infty a_n\cos(n\theta)$ is the Fourier cosine series of $f(2\cos\theta)$ and hence converges to $f(2\cos\theta)$ in $L^2$. Hence $2a_0 + \sum_{n=1}^\infty a_nP_n(x)$ converges to $f(x)$ in $L^2_w$. Now we are ready to prove the following:
\begin{lemma} \label{lem:seminormsEquiv}
  For $f\in L^\infty([-2, 2])$, the $V_{\T{GUE}}$ seminorm and Chebyshev-$\dot{H}^\frac{1}{2}$ seminorm are equivalent in the sense that:
    \begin{equation}
      4 V_{\T{GUE}}[f] = \norm{f}_{\dot{H}^\frac{1}{2}_C}
    \end{equation}
\end{lemma}
\begin{proof}
We first show that $\norm{f}_{\dot{H}^\frac{1}{2}_C}\ge 4 V_{\T{GUE}}[f]$. We begin by noting that for smooth $f(x)$, using integration by parts:
  \begin{equation*}
  \begin{split}
    -\int_{-2}^2 \frac{f'(y)\sqrt{4-y^2}}{y-x} dy & = \frac{-(f(y) - f(x))\sqrt{4-y^2}}{y-x}\bigg \vert_{-2}^2 + \int_{-2}^2 \frac{(f(x) - f(y))(4-xy)}{\sqrt{4-y^2}(x-y)^2} dy \\
    & = \int_{-2}^2 \frac{(f(x) - f(y))(4-xy)}{\sqrt{4-y^2}(x-y)^2} dy.
  \end{split}
  \end{equation*}
Hence, using symmetry in $x$ and $y$ we have that:
  \begin{equation*}
    -\frac{1}{4\pi^2}\int_{-2}^2\frac{f(x)}{\sqrt{4-x^2}}\int_{-2}^2 \frac{f'(y)\sqrt{4-y^2}}{y-x} dx dy = \frac{1}{8\pi^2}\int_{-2}^2\int_{-2}^2 F(x, y) \frac{4-xy}{\sqrt{4-x^2}\sqrt{4-y^2}} dxdy.
  \end{equation*}
Note that $\frac{d}{dy}P_n(y) = \frac{n}{2}U_{n-1}(y/2)$, $n \ge 0$ where $U_n$ is the $n$th Chebyshev polynomial of the second kind. In particular we have the identity:
  \begin{equation*}
    \frac{1}{2\pi}\int_{-2}^2 \frac{U_{n-1}(y/2) \sqrt{4-y^2}}{y-x} dy = -P_n(x).
  \end{equation*}
Then if we can write $f^N(x) = \sum_{n=0}^N a_n P_n(x)$, $\abs{x} \le 2$, we see that,
  \begin{equation*}
    -\frac{1}{4\pi^2}\int_{-2}^2\frac{f^N(x)}{\sqrt{4-x^2}}\int_{-2}^2 \frac{f^{N\prime}(y)\sqrt{4-y^2}}{y-x} dx dy = \frac{1}{8} \sum_{n=1}^N na_n^2.
  \end{equation*}
This is part of the well-known result of Johansson \cite{Johansson1998}. In particular, using Fatou's lemma we have the inequality,
  \begin{equation*}
  \begin{split}
    \liminf_{N} \frac{1}{8} \sum_{n=1}^N na_n^2 \ge \frac{1}{8\pi^2}\int_{-2}^2 \int_{-2}^2 \liminf_N \Big(\frac{\sum_{n=1}^Na_n (P_n(x) - P_n(y))}{x-y}\Big)^2 \frac{4-xy}{\sqrt{4-x^2}\sqrt{4-y^2}} dx dy.
  \end{split}
  \end{equation*}
So if $f(x) = \sum_{n=0}^\infty a_n P_n(x)$ then,
  \begin{equation} \label{eq:ChebyshevBigGUE}
    \sum_{n=1}^\infty na_n^2 \ge 4 V_{\T{GUE}}[f].
  \end{equation}
We now work on showing $4 V_{\T{GUE}}[f] \ge \norm{f}_{\dot{H}^\frac{1}{2}_C}$. For the remainder of the proof, for conciseness, we will use $\kappa_2(x) := \kappa(x+2)= 2\pi\sqrt{4-x^2}$ from \eqref{eq:kappaDefs}. Because of the orthogonality of $P_n$, we have:
  \begin{equation*}
   \int_{-2}^2 K_t(x,y) \frac{1}{\kappa_2(y)} dy = 1.
  \end{equation*}
Thus we can write:
  \begin{equation*}
   \inn{f}{f}_{w} = \frac{1}{2}\int_{-2}^2\int_{-2}^2 K_t(x,y) \frac{f(x)^2 + f(y)^2}{\kappa_2(x)\kappa_2(y)} dx dy.
  \end{equation*}
Now for any fixed $t > 0$, Lemma \ref{lem:KKernal} tells us that $K_t(x, y)$ is bounded. As a result, Cauchy-Schwarz tells us that $\inn{f}{e^{-tK}f}_w < \infty$ for $f\in L^2_w([-2, 2])$. Hence the following is well-defined for $t > 0$:
  \begin{equation} \label{eq:KtDividet}
   I_t(f) := \frac{1}{t}(\inn{f}{f}_w - \inn{f}{e^{-tK}f}) = \frac{1}{2}\int_{-2}^2\int_{-2}^2 \frac{K_t(x,y)}{t} \frac{(f(x) - f(y))^2}{\kappa_2(x)\kappa_2(y)} dx dy.
  \end{equation}
From the orthogonality of $P_n$ we also have the equality, 
  \begin{equation*}2\int_{-2}^2 \frac{f(x)}{\kappa_2(y)} dy + 4 \sum_{n=1}^\infty \int_{-2}^2 \frac{f(x)P_n(x)P_n(y)}{\kappa_2(y)}dy = f(x).
  \end{equation*}
In particular, this gives:
  \begin{equation*}
  \begin{split}
   \inn{f}{f}_w & = \int_{-2}^2 \int_{-2}^2 \frac{f(x)^2 + f(y)^2}{\kappa_2(x)\kappa_2(y)} dx dy + 2 \sum_{n=1}^\infty \int_{-2}^2 \int_{-2}^2 \frac{f(x)^2 + f(y)^2}{\kappa_2(x)\kappa_2(y)} P_n(x)P_n(y)dx dy \\
    & \ge 2\int_{-2}^2 \int_{-2}^2 \frac{f(x)f(y)}{\kappa_2(x)\kappa_2(y)} dx dy + 4 \sum_{n=1}^\infty \int_{-2}^2 \int_{-2}^2 \frac{f(x)f(y)}{\kappa_2(x)\kappa_2(y)} P_n(x)P_n(y)dx dy.
  \end{split}
  \end{equation*}
At the same time Fubini's theorem tells us that,
  \begin{equation*}
  \begin{split}
   \inn{f}{e^{-tK}f}_w = 2\int_{-2}^2\int_{-2}^2 & \frac{f(x)f(y)}{\kappa_2(x)\kappa_2(y)} dx dy + 4\sum_{n=1}^\infty e^{-\frac{tn}{2}}\int_{-2}^2\int_{-2}^2 \frac{f(x)f(y)}{\kappa_2(x)\kappa_2(y)}P_n(x)P_n(y) dx dy.
   \end{split}
  \end{equation*}
This implies that,
  \begin{equation*}
  \begin{split}
   I_t(f) = \frac{1}{t}(\inn{f}{f}_w - \inn{f}{e^{-tK}f}) & \ge 4 \sum_{n=1}^\infty \frac{1-e^{-\frac{tn}{2}}}{t} \int_{-2}^2\int_{-2}^2 \frac{f(x)f(y)}{\kappa_2(x)\kappa_2(y)}P_n(x)P_n(y) dx dy \\
   & = \frac{1}{8} \sum_{n=1}^\infty \frac{2(1-e^{-\frac{tn}{2}})}{t} a_n^2,
  \end{split}
  \end{equation*}
where $a_n = \int_{-2}^2 \frac{4f(x)P_n(x)}{\kappa_2(x)}dx$ is the Chebyshev coefficient as defined before. In particular, we have:
  \begin{equation*}
   \sup_{t > 0}I_t(f) \ge \sup_{t > 0} \frac{1}{8} \sum_{n=1}^\infty \frac{2(1-e^{-\frac{tn}{2}})}{t} a_n^2 = \frac{1}{8} \sum_{n=1}^\infty n a_n^2.
  \end{equation*}
Then monotone convergence with Lemma \ref{lem:KKernalNegative} tells us that,
  \[\sup_{t > 0}I_t(f) = \lim_{t\to 0}I_t(f) = \inn{f}{Kf}_w = \frac{1}{2}V_{\T{GUE}}(f),\]
where the last equality holds by the definition of $K$ in \eqref{defs:KOperator}. So for $f \in L^2_w([-2, 2])$,
  \begin{equation*}
   4V_{\T{GUE}}(f) \ge \sum_{n=1}^\infty n a_n^2.
  \end{equation*}
Hence the two seminorms are equal.
\end{proof}

\subsection{Proof of Lemma \ref{lem:VLUEApprox}}
This is simply an application of Lemma \ref{lem:seminormsEquiv} with a change of variables.
\begin{proof}
Let $f\in L^\infty([0, \infty))$ be given and suppose \eqref{eq:Assumption1} holds. Performing a change of variable in $\tilde{x} = x-2$ and $\tilde{y} = y-2$:
  \begin{equation*}
    V_{\T{LUE}}[f] = \frac{1}{4\pi^2} \int_{-2}^2 \int_{-2}^2 \Big(\frac{f(\tilde{x} + 2)-f(\tilde{y} + 2)}{\tilde{x} - \tilde{y}}\Big)^2 \frac{4-\tilde{x}\tilde{y}}{\sqrt{4-\tilde{x}^2}\sqrt{4-\tilde{y}^2}} d\tilde{x} d\tilde{y}.
  \end{equation*}
Hence $V_{\T{LUE}}[f] = V_{\T{GUE}}[\tilde{f}]$ where $\tilde{f}(x) = f(x+2)|_{x \in [-2, 2]}$. Let $a_k$ be the Chebyshev coefficients of $\tilde{f}(x)$ and $\tilde{f}^N(x) = \sum_{n=0}^N a_n P_n(x)$. Then Lemma \ref{lem:seminormsEquiv} tells us that
  \begin{equation} \label{eq:ChebyshevApproxVariance}
  \begin{split}
    \sum_{n=N+1}^\infty n a_n^2 = 4V_{\T{GUE}}[\tilde{f}-\tilde{f}^N].
  \end{split}
  \end{equation}
Hence $\tilde{f}^N$ can be made such that is approximates $\tilde{f}$ arbitrarily well in the $V_{\T{GUE}}$ seminorm. So given $\nu > 0$, we can find sufficiently large $N$ such that $V_{\T{GUE}}[\tilde{f}-\tilde{f}^N] < \nu$. Define $f^N(x) = \tilde{f}^N(x-2)$ for $x \in [-2, 2 + \epsilon]$ for some $\epsilon > 0$ and extend it to $[0,\infty)$ smoothly with bounded derivative. Then,
  \[V_{\T{LUE}}[f - f^N] = V_{\T{GUE}}[\tilde{f} - \tilde{f}^N] < \nu.\]
\end{proof}
\newpage
\begingroup
\setlength\bibitemsep{8pt}
\printbibliography
\endgroup

\end{document}